%% file: main-hal-v2.tex
\documentclass[10pt]{scrartcl}

\usepackage[utf8]{inputenc}
\usepackage{hyperref}
\hypersetup{
pdftitle={Asymptotics of motion planning complexity for control-affine systems},
pdfauthor={Michele Motta, Dario Prandi},
}

\usepackage[english]{babel}
\usepackage{amsmath}
\usepackage{amssymb}
\usepackage{amsthm}
\usepackage{amsfonts}
\usepackage{mathtools}
\usepackage{makeidx}
\usepackage{etoolbox}
\usepackage{bbm}
\usepackage{subfiles}
\usepackage{graphicx}
\usepackage{svg}
\usepackage{dsfont}
\usepackage{subcaption}
\usepackage{calc}
\usepackage{accents}
\usepackage[normalem]{ulem}

\usepackage{pgfplots}
\usepackage{tikz}
\usepackage{tikz-3dplot}
\usetikzlibrary{intersections}

\usepackage{xcolor}
\usepackage{pdflscape}

\newtheorem{thm}{Theorem}[section]

\newtheorem{prop}[thm]{Proposition}
\newtheorem{lem}[thm]{Lemma}

\newtheorem{claim}[thm]{Claim}

\theoremstyle{definition}
\newtheorem{defn}[thm]{Definition}

\theoremstyle{remark}
\newtheorem{rem}[thm]{Remark}

\pgfplotsset{compat=1.15}
\usetikzlibrary{3d}
\usetikzlibrary{shapes}

\pgfplotsset{compat=newest}

\numberwithin{equation}{section}

\newcommand{\R}{\mathbb{R}}

\newcommand{\N}{\mathbb{N}}

\newcommand{\al}{\alpha}
\newcommand{\be}{\beta}
\newcommand{\e}{\eta}

\newcommand{\eps}{\varepsilon}

\newcommand{\pa}{\partial}

\DeclareMathOperator{\MCeps}{MC_\varepsilon}
\DeclareMathOperator{\MC}{MC}
\DeclareMathOperator{\rmc}{rmc}
\DeclareMathOperator{\armc}{a-rmc}

\newcommand{\vardbtilde}[1]{\tilde{\raisebox{0pt}[0.85\height]{$\tilde{#1}$}}}

\title{Asymptotics of motion planning complexity for control-affine systems}

\author{Michele Motta \thanks{SISSA, Trieste}, Dario Prandi\thanks{Université Paris-Saclay, CNRS, CentraleSupélec, Laboratoire des Signaux et Systèmes, 91190, Gif-sur-Yvette, France.}}

\usepackage{todonotes}

\begin{document}

\maketitle

\begin{abstract}
    In this paper, we study the complexity of the approximation of nonadmissible curves for nonlinear control-affine systems satisfying the strong Hörmander condition.
    Focusing on tubular approximation complexities, we provide asymptotic equivalences, with explicit constants, for all generic situations where the distribution, i.e., the linear part of the control system, is of co-rank one.
    Namely, we consider curves in step $2$ distributions and any dimension. In the $3$ dimensional case, we also consider the case of distributions with Martinet-type singularities that are crossed by the curve at isolated points.
\end{abstract}

\tableofcontents

\section{Introduction}

Given a control system on a  manifold $M$ of the form
\begin{equation*}
    \label{eq:3}
    \dot q = f(q,u), \qquad q\in M, \, u\in \mathbb{R}^m,
\end{equation*}
the motion planning problem consists in finding an admissible trajectory $q_u$ steering the system between a given initial position and a final one, usually under further requirements, such as obstacle avoidance.
One possible approach to solve such a problem is to split it into two different tasks:
(i) find a (usually non-admissible) trajectory $\Gamma$ solving the problem (i.e., connecting the two given points without touching the obstacles),
and (ii) track $\Gamma$ by means of admissible trajectories.
Since the first step depends only on the topology of the manifold and of the obstacles, we focus on the second one, which depends only on the local nature of the control system near the path.

In order to measure the precision of the tracking, we introduce a distance $d:M\times M\to [0,+\infty)$ and consider that a  trajectory $q_u$ tracks $\Gamma$ within precision $\varepsilon$ if $q_u([0,T])\subset \operatorname{Tube}_{\varepsilon}(\Gamma)$, where we let
\begin{equation*}
    \operatorname{Tube}_{\varepsilon}(\Gamma) = \left\{ p \in M \mid d(\Gamma,p)\le \varepsilon\right\}.
\end{equation*}
Moreover, we assume that a cost function $J$ is provided, and look for trajectories solving step (ii) in time $T>0$ while minimizing $J$.
This is quantified by the $\varepsilon$-complexity associated with a compact curve $\Gamma \subset M$ connecting $q_1,q_2\in M$ , which for $\varepsilon>0$ is the quantity
\begin{multline}
    \label{eq:mc-def}
    \MC_\varepsilon(\Gamma,T)
    =
    \frac1\varepsilon
    \inf\bigg\{
    J(u)
    \mid
    u \in L^\infty\left( [0,T];\mathbb{R}^m \right),\\
    q_u([0,T]) \subset \operatorname{Tube}_{\varepsilon}(\Gamma),
    \,
    q_u(0)=q_1,
    \,
    q_u(T)=q_2
    \bigg\}.
\end{multline}
Clearly, when $\Gamma$ is not an admissible trajectory of the control system, the complexity diverges as $\varepsilon$ tends to $0$ and the aim of this paper is to precisely characterise its asymptotic behaviour in a specific class of control-affine systems.

In the case of a sub-Riemannian (or nonholonomic) control system, i.e., of a control system linear w.r.t.\ the control, a natural choice for the cost is the $L^1$-norm of the controls.
Due to the reversibility in time of such a system, the associated value function is in fact a distance, called Carnot-Carath\'eodory distance, that endows $M$ with a metric space structure.
In this context various definitions of complexity have been proposed (see \cite{Bellaiche1996,Jean2001}), and asymptotic estimates, up to multiplicative constants, are known (see \cite{Jean2003a}).
We mention also that for a restricted set of sub-Riemannian systems, i.e. one-step bracket generating or with two controls and dimension $\le 6$, precise asymptotic estimates and explicit asymptotic optimal syntheses are obtained in the series of papers \cite{c1,9,10,11,12,13,14,Boizot} (see  \cite{gauthierbonnard} for a review). We also mention \cite{gauthierKawskiz} where approximate optimal synthesis are derived.

In \cite{c2}, the results of \cite{Jean2003a} have been partially extended to the control-affine case. Namely, the authors define different notions of complexities for systems of the form
\begin{equation}
    \dot q = X_0(q) + \sum_{j=1}^m u_j X_j(q), \qquad u\in \mathbb{R}^m,
\end{equation}
and then provide the corresponding rate of divergence under the strong Hörmander condition.
\subsection{Setting and main results}
Let $M$ be a smooth and connected manifold of dimension $n\ge 3$, on which is given a family of smooth vector fields $X_0,X_1,\ldots, X_m\in \operatorname{Vec}(M)$, $m\in \mathbb{N}^*$, defining a control system of the form
\begin{equation}
    \label{eq:cs-drift}
    \dot q = X_0(q) + \sum_{j=1}^m u_j X_j(q), \qquad u\in \mathbb{R}^m.
\end{equation}
Throughout the paper, we assume that the distribution $\Delta = \operatorname{span}\{X_1,\ldots, X_m\}$ satisfies the following assumptions:
\begin{itemize}
    \item[\textbf{(H0)}] Strong Hörmander condition: for any $q\in M$ it holds
          \begin{equation}
              \operatorname{Lie}\Delta(q) = T_q M.
          \end{equation}
    \item[\textbf{(H1)}]\label{hyp:corank} Co-rank one: $\dim\Delta(q) = n-1$ for any $q\in M$.
\end{itemize}
Due to assumption \textbf{(H0)}, the control-linear part of system \eqref{eq:cs-drift} defines a sub-Riemannian structure on $M$. We denote by $g$ the associated metric and $|\cdot|_g$ the corresponding norm, which is defined by
\begin{equation}
    \label{eq:metric}
    |v|_g = \min\left\{ |u| \mid \sum_{j=1}^m u_j X_j(q) = v \right\},
    \qquad \forall q\in M, \, v\in \Delta(q).
\end{equation}
For simplicity of notation, we set $|v|_g = +\infty$ whenever $v\notin \Delta(q)$.
It is standard to observe that the metric complexity problem with cost $J(u) = \|u\|_{L^1}$ for system \eqref{eq:cs-drift} depends only on the couple $(\Delta,g)$ and not on the precise choice of frame $\{X_1,\ldots, X_m\}$.

In this work, we follow the path started in \cite{c1} and derive precise asymptotics for the complexity \eqref{eq:mc-def} of system \eqref{eq:cs-drift}. To set up the motion planning problem we need a curve $\Gamma$ that we assume to satisfy the following
\begin{itemize}
    \item[\textbf{(H2)}] The smooth curve $\Gamma\subset M$ is oriented, compact with non-empty boundary, and such that\footnote{
              The assumption $T \Gamma\pitchfork \Delta$ is present just to simplify the exposition. See \cite{c1}.} $T\Gamma \pitchfork \Delta$.
\end{itemize}
Henceforth, we let $T^+\Gamma\subset T\Gamma$ be the set of positively oriented tangent vectors to $\Gamma$
and we label the boundary points of $\Gamma$ as $\partial\Gamma=\{q_1,q_2\}$ in a way that is coherent with the orientation of $\Gamma$.
Fixing an orientation for $\Gamma$ is essential, since system \eqref{eq:cs-drift} is not reversible. Indeed, we need to distinguish whether the drift $X_0$ is working in our favour or not.
To this aim, we introduce the following.

\begin{defn}
    A one-form $\omega$ is associated with $(\Delta,\Gamma)$, where $\Gamma$ is an oriented curve, if $\Delta = \ker\omega$ and $\omega(v)>0$ for any $v\in T^+\Gamma$.
\end{defn}

Observe that the existence of forms associated with $(\Delta,\Gamma)$ is guaranteed by assumption \textbf{(H1)}. Given any form associated with $(\Delta,\Gamma)$, say $\omega$, the set of all forms associated with $(\Delta,\Gamma)$ is given by\footnote{Indeed, all possible forms with the same kernel as $\omega$ are obtained by multiplication by a non-vanishing smooth function $\varphi$. To preserve the required property w.r.t.~the orientation of $\Gamma$ we need $\varphi|_\Gamma>0$, which then yields $\varphi>0$.} $\{\varphi \omega \mid \varphi\in C^\infty(M),\, \varphi>0\}$.

Since $T_qM = \Delta(q)\oplus T_q\Gamma$ for any $q\in \Gamma$, we have the decomposition
\begin{equation}
    \label{eq:drift-decomposition}
    X_0(q)  = X_0^\Delta(q) + X_0^\Gamma(q), \
    \text{ with }  X_{0}^\Delta(q)\in \Delta, \,  X_{0}^\Gamma(q)\in T\Gamma,
    \qquad \forall q\in \Gamma,
\end{equation}
Then, we define
\begin{equation}
    \label{eq:T-gamma}
    T_\Gamma := \inf \{ t>0 \mid e^{t X_0 ^\Gamma}q_1 = q_2 \}.
\end{equation}
We say that $X_0$ and $\Gamma$ are co-oriented if $T_\Gamma < +\infty$. Observe that this is equivalent to the fact that $\omega(X_0)>0$ on $\Gamma$ for any form $\omega$ associated with $(\Delta,\Gamma)$.

In the first case that we consider, we make also the following assumption. 
\begin{itemize}
    \item[\textbf{(H3)}] For any $q\in \Gamma$, the distribution $\Delta$ is of step 2, that is 
    \begin{equation}
        \label{eq:step2}
        \dim \mathrm{span} \{\Delta(q)+[\Delta,\Delta](q)\} = n.
    \end{equation}
\end{itemize}
Our first result is the following.

\begin{thm}[Step 2 case]
    \label{thm:main}
    Assume that $\Delta$ and $\Gamma$ are a distribution and a curve on a smooth manifold $M$ satisfying \textbf{(H1)}, \textbf{(H2)} and \textbf{(H3)}.
    Let $\omega$ be a form associated with $(\Delta,\Gamma)$ and define
    \begin{equation}
        \label{eq:alpha}
        \alpha(q) \coloneq \max \big\{ \omega_q \big([X_1,X_2]\big) \mid X_1,X_2\in \Delta,\,|X_1|_g = |X_2|_g = 1 \text{ in } \mathcal{V}_q \big\}.
    \end{equation}
    Here, $\mathcal{V}_q$ is any neighbourhood of $q$.
    Then, recalling the definition of $T_\Gamma$ in \eqref{eq:T-gamma}, we have:
    \begin{itemize}
        \item if $T_\Gamma<T$:
              \begin{equation}
                  \label{eq:main-thm1}
                  \operatorname{MC}_\varepsilon(\Gamma,T) \simeq \frac{2(T-T_\Gamma)}{\varepsilon^2} \min_{q\in\Gamma} \frac{\omega(X_0(q))}{\alpha(q)};
              \end{equation}
        \item if $T_\Gamma=T$:
              \begin{equation}
                  \label{eq:main-thm2}
                  \operatorname{MC}_\varepsilon(\Gamma,T) \simeq \frac1\varepsilon \int_\Gamma \frac{| X_0^\Delta|}{\omega(X_0)}\, \omega,
              \end{equation}
              where $X_0^\Delta$ is defined in \eqref{eq:drift-decomposition};
        \item if $T_\Gamma>T$: there exists $\Omega \subset \Gamma$ defined in Lemma~\ref{lem:I^*}, such that $|\Gamma \setminus\Omega|>0$.
        If $|\Omega|>0$, then 
              \begin{equation}
                  \label{eq:main-thm3}
                  \operatorname{MC}_\varepsilon(\Gamma,T) \simeq \frac{2}{\varepsilon^2} \int_{\Gamma\setminus\Omega} \frac{1}{\alpha} \omega.
              \end{equation}
        If, instead, $|\Omega|=0$, then
        \begin{equation}
            \label{eq:main-thm4}
            \MCeps(\Gamma,T) 
            \simeq 
            \frac{2}{\eps ^2} 
            \left(
                \int_\Gamma \frac{1}{\al}\omega - T \min_{q\in \Gamma} \frac{|\omega(X_0)|}{\al(q)}
            \right). 
        \end{equation}
    \end{itemize}
\end{thm}

The three cases of the theorem can be interpreted as follows:
\begin{enumerate}
   \item The drift is too strong along $T^+\Gamma$, and thus we have to ``slow down'' the motion. The optimal strategy is to do this when the ratio between the component of the drift along $T\Gamma$ and the maximal velocity at which the system can travel in the $T\Gamma$ direction is minimal.
   \item The component of the drift along $T\Gamma$ steers the system exactly to the desired endpoint, so it is sufficient to neutralize the components of the drift along $\Delta$.
   \item The component of the drift along $T^+\Gamma$ is either too weak or direct in the opposite direction, so we have also to ``push'' the system towards the desired endpoint. 
   In this case, the set $\Omega$ corresponds to the parts of $\Gamma$ where the drift is ``helping'' to reach the endpoint.
    Hence, on $\Gamma\setminus\Omega$ we have to push the system, while on $\Omega$ we can simply follow the flow of the drift.
    In particular, if $\Omega$ has zero measure, we have to push the system along the whole curve $\Gamma$, and the additional term in Equation \eqref{eq:main-thm4} takes into account the final time constraint.
\end{enumerate}
\begin{rem}
    Notice that, if $\Delta$ is of step 2 along $\Gamma$, then at every point $q\in\Gamma$ there are $X_1,X_2$ tangent to $\Delta$ such that $[X_1,X_2](q) \pitchfork \Delta(q)$, hence $[X_1,X_2](q)\neq 0$ and in particular $\al(q)\neq0$.
\end{rem}

\begin{rem}
    By the previous result, when $T_\Gamma\neq T$, we have that $\operatorname{MC}_\varepsilon(\Gamma,T) \simeq {C({\Gamma,T})}{\eps^{-2}}$,
    for an explicit constant $C(\Gamma,T)$. The order of the asymptotic drastically changes when $T=T_\Gamma$, becoming ${\eps^{-1}}$.
    This is coherent with the fact that, in the case where $T_\Gamma<+\infty$, $C(\Gamma,T)\to 0$ as $T\to T_\Gamma$.
\end{rem}

\begin{rem}
    Letting $X_0=0$ we recover the result of \cite{c1}. Indeed, in this case $T_\Gamma=+\infty$, and $\Omega=\varnothing$.
\end{rem}

We stress that the technique of proof for Theorem~\ref{thm:main} consists in explicitly exhibiting an asymptotic optimal synthesis.
Although similar in spirit to the synthesis considered in \cite{c1} for the control-linear case, the presence of a drift introduces some non-trivial technical difficulties, for example in showing that controls driving the system ``against the drift'' cannot be asymptotically optimal when the drift pushes the system in the direction of $\Gamma$ (Theorem~\ref{thm:sign-of-control}).

\begin{rem}
    \label{rem:gen-distrib-dim-ge-4}
    It turns out that if $n\ge 4$, the assumption \textbf{(H3)} of the above theorem is actually generic among corank $1$ distributions.
    More precisely, letting $\mathfrak X$ be the set of $(\Delta, g,\Gamma)$ satisfying
    the co-rank one assumption \textbf{(H1)},
    the set $\mathfrak X_*\subset \mathfrak X$ containing only those distributions that satisfy \textbf{(H3)} is open and dense in the $C^\infty$ topology. See, \cite{c1}.
\end{rem}

Since this generically does not happen in dimension $n\ge 4$, we restrict our attention to the case $n=3$.
Here, we focus only on the following situation, which is generic \cite{c1}.

\begin{defn}
    \label{def:generic-type}
    For $n=3$, we say that the couple $(\Delta,\Gamma)$ is of \emph{generic type} if $\Gamma$ satisfies \textbf{(H2)} and there exists a \emph{finite} set of points $\Sigma\subset \Gamma$ (possibly empty), called \emph{Martinet points}, on which $\Delta$ is step 3, while it is step 2 on $\Gamma\setminus\Sigma$. More precisely, for any $q\in\Sigma$,
    \begin{gather}
        \dim \mathrm{span} \{\Delta(q)+[\Delta,\Delta](q)\}=2,
        \\[5pt]
        \dim \mathrm{span}\big\{ \Delta(q)+[\Delta,\Delta](q) + [\Delta,[\Delta,\Delta]](q) \big\} = 3.
    \end{gather}
    Moreover, we require that, at every Martinet point $q\in\Sigma$ it holds that $d_q\al\mid_{\Delta(q)} \neq 0$.
\end{defn}

Near Martinet points tracking $\Gamma$ can become more expensive, and indeed in the sub-Riemannian case they dominate the asymptotic \cite{c1}. In presence of a drift, however, this is not always true: if the drift is co-oriented with $\Gamma$ at Martinet points, we can exploit it to bypass these points.
We then have the following result.

\begin{thm}
    \label{thm:main2}
    Assume that $n=3$ and that $(\Delta,\Gamma)$ is of generic type.
    Let $\omega$ be a form associated with $(\Delta,\Gamma)$ and let $\Sigma^-=\{ q_i\in \Sigma \mid \omega(X_0(q_i)) \le 0 \}$. We have the following :
    \begin{itemize}
        \item If $\Sigma^-=\varnothing$, the same conclusions as in Theorem~\ref{thm:main} holds.
        \item If $\Sigma^- = \{q_1,\ldots, q_r\}$, then
              \begin{equation}
                  \operatorname{MC}_\varepsilon(\Gamma,T) \simeq -\frac{\log\varepsilon}{\varepsilon^2}\sum_{i=1}^r \frac{4}{\kappa_i}.
              \end{equation}
              Here, $\kappa_1,\ldots,\kappa_r$ are defined by
              \begin{equation}
                  \label{eq:def-k_i}
                  \kappa_i
                  =
                  \big|
                  \omega([W_\Gamma,[X_1,X_2]](q_i))-d\omega(W_\Gamma,[X_1,X_2](q_i))
                  \big|,
                  \qquad i=1,\ldots,r,
              \end{equation}
              where $\{X_1,X_2\}\subset\operatorname{Vec}(M)$ is a local orthonormal frame for the metric $g$
              near $q_i$,
              and $W_\Gamma$ is any vector field such that $\omega(W_\Gamma)=1$ and $W_\Gamma|_\Gamma\subset T\Gamma$.
    \end{itemize}
\end{thm}

\begin{rem}
    Replacing $\omega$ with another form associated with $(\Delta,\Gamma)$ does not change the value of $\kappa_i$, as can be checked by direct computations.
    Moreover, in Proposition~\ref{lemma:ciao} we show that the above definition of $\kappa_i$ coincides with the one provided in \cite{c1}, that is
    \begin{equation}
        \kappa_i = d\alpha(W_\Gamma(q_i)).
    \end{equation}
    In particular, the above result generalises the result of that reference.
\end{rem}

The paper is organized as follows: in Section \ref{sec:Preliminaries} we prove some preliminary results and we prove the asymptotic in Equation \eqref{eq:main-thm2};
in Section \ref{sec:rmc-equiv} we introduce the notion of \emph{reduced motion complexity} and we show how the problem of finding an asymptotic for $\MCeps$ can be reduced to the study of this new simplified quantity, see Subsection \ref{subsec:proof-of-red};
then, in Subsection \ref{subsec:prop-of-rmc} we prove some important properties of the reduced motion complexity which will be used in the proof of Theorem \ref{thm:main};
in Section \ref{sec:proof-step-2} we prove Theorem \ref{thm:main};
finally, in Section \ref{sec:Martinet} we prove Theorem \ref{thm:main2}: we first deal with the case a single Martinet point (see \ref{subsec:sing-martinet-point}) and then we show how the general case can be reduced to the case of one single point.

\section{Preliminaries and normal forms}
\label{sec:Preliminaries}
In this Section we first prove a continuity result for the quantity $\MCeps(\Gamma,T)$.
In particular, this leads to a simple proof of Equation \eqref{eq:main-thm2}.
Then in Subsection \ref{subsec:norm-coord}, we introduce normal coordinates in $\operatorname{Tube}_\eps (\Gamma)$.

\subsection{Continuity of motion complexity}
Since we are going to consider motion complexities associated to different control systems, only in this Subsection we specify the dependence of the motion complexity on the vector fields $X_0,X_1,\dots,X_m$: using the multi-index notation $X_I = (X_0,X_1,\dots,X_m)$, we write $\MC_\varepsilon(\Gamma,T,X_I)$. 

\begin{lem}
    \label{lem:continuity-MC}
    Let $(X_I^j)_{j\in\N}$ be a sequence of $(m+1)$-tuple of smooth vector fields on $M$ converging to $X_I ^\infty$ in the product topology of $\operatorname{Vec}(M)^{m+1}$ induced by the $C^0$ topology of $\operatorname{Vec}(M)$.
    Then, for every $\eps>0$, we have
    \begin{equation}
        \lim_{j\to+\infty} \MC_\varepsilon(\Gamma,T,X_I^j) = \MC_\varepsilon (\Gamma,T,X_I ^\infty)
    \end{equation}
\end{lem}
\begin{proof}
    We begin by showing the following inequality:
    \begin{equation}
        \label{eq:ineq1-continuity-drifts}
        \limsup_{j\to+\infty} \MC_\varepsilon(\Gamma,T,X_I^j)
        \leq
        \MC_\varepsilon (\Gamma,T,X_I ^\infty).
    \end{equation}
    By definition of $\MC_\varepsilon (\Gamma,T,X_I^\infty)$, for every $\delta>0$ there is a control $u=(u_1,\dots,u_m)$ admissible for $\MC_\varepsilon (\Gamma,T,X_I^\infty)$ such that
    \begin{equation}
        \frac{1}{\eps} \|u\|_1 \leq \MC_\varepsilon (\Gamma,T,X_I^\infty) + \delta.
    \end{equation}
    Let us denote by $q_\infty$ the trajectory associated to $u$.
    Now, for each $j\in\N$, consider the trajectory $q_j$
    of the control system \eqref{eq:cs-drift} with drift $X_0 ^j$ and control $u$.
    Then, standard arguments show that the $C^0$ convergence of $X_I^j$ to $X_I ^\infty$ yields the uniform convergence of $q_j$ to $q_\infty$ on $[0,T]$.

    Since $q_j \to q_\infty$ uniformly, we have that, for every $j$ big enough, there is some $\tau_j>0$ such that
    \begin{equation}
        |q_j(T-\tau_j)-q_\infty(T)| \leq \delta,
    \end{equation}
    where here $|\cdot|$ denotes the euclidean distance in a suitable local chart.
    Then, since for the driftless system the $L^1$ cost to join  $q_{j}(T-\tau_j)$ to $q_\infty(T)$ is bounded by $C\delta^{1/2}$ (this is the Ball-Box Theorem, see \cite{Bellaiche1996} or \cite[Corollary~2.1]{jeanControl2014}), a standard time-rescaling argument (see also  \cite[Theorem~1.1]{prandi2013}) yields that
    we can reach $q_\infty(T)$ from $q_j(T-\tau_j)$ in time $\tau_j$ with a control $v_j$, whose $L^1$-norm is bounded by
    \begin{equation}
        \label{eq:ineq-cont-drift-controllo-tempo-finale}
        \frac{1}{\eps} \| v_j \|_1
        \leq
        C \left( \delta^{1/2} + \tau_j^{1/2}\right).
    \end{equation}
    Since the concatenation of the trajectory ${q_j}_{|[0,T-\tau_j]}$ and the trajectory given by $v_j$ is admissible for $\MCeps(\Gamma,T,X_I^j)$, we have that
    \begin{align}
        \MCeps(\Gamma,T,X_I^j)
         & \leq
        \frac{1}{\eps}
        \int_0 ^{T-\tau_j} |u(t)| dt
        +
        \frac{1}{\eps}
        \int_{T-\tau_j} ^T |v_j(t)| dt
        \leq
        \\
         & \leq
        \MC_\varepsilon (\Gamma,T,X_I^\infty) + \delta
        +
        C \left( \delta^{1/2} + \tau_j^{1/2}\right).
    \end{align}
    Notice that, as $j\to+\infty$, we can take $\tau_j\to 0$. Thus, taking the $\limsup$ in $j$, we obtain the inequality in \eqref{eq:ineq1-continuity-drifts}.

    Now, we want to prove the opposite inequality with $\liminf$ in place of $\limsup$.
    As before, we fix $\delta>0$ and for every $j\in\N$ we take a control $u_j$ admissible for $ \MCeps(\Gamma,T,X_I^j) $ such that
    \begin{equation}
        \label{eq:ineq2-cont-drift}
        \frac{1}{\eps} \|u_j\|_1
        \leq
        \MC_\varepsilon (\Gamma,T,X_I^j) + \delta.
    \end{equation}
    Let us denote by $q_j$ the trajectory associated to the control $u_j$.
    By \eqref{eq:ineq1-continuity-drifts} and \eqref{eq:ineq2-cont-drift}, we have that $\sup_j \|u_j\|_1 < +\infty$.
    We claim that, up to a subsequence, the $q_j$ converges uniformly to a limit trajectory $q_\infty$.
    Indeed, all the $q_j$ are contained in $\operatorname{Tube}_\eps(\Gamma)$, hence they are equibounded.
    Moreover, they are equilipschitz:
    \begin{equation}
        |q_{j}(t_1) - q_{j}(t_2)|
        =
        \left|
        \int_{t_1} ^{t_2} X_0^{j}(q_{j}) + \sum_{i=1}^m u_{j}(s) X_i(q_{j}) ds
        \right|
        \leq
        \left(
        C_1+ \|u_j\|_1 C_2
        \right)
        |t_1-t_2|
    \end{equation}
    where $C_1=\sup_{j\in \N, \, q\in\operatorname{Tube}_\eps(\Gamma)} |X_0 ^j(q)|$
    and $C_2=\max_{i\in{1,\dots,m}, \, q\in\operatorname{Tube}_\eps(\Gamma)}\left|X_i(q)\right|$.
    Thus, by Arzelà-Ascoli Theorem, the trajectories $q_j$ admit a converging subsequence and we denote by $q_\infty$ its limit.
    Now, arguing as in Theorem 3.41 in \cite{ABB}, it is not difficult to show that $q_\infty$ is admissible for the problem $\MCeps(\Gamma,T,X_I^\infty)$.
    We point out two differences between our case and the one treated in the reference:
    first, they take into consideration only control-linear systems, but their proof does not require any particular structure of the control system;
    second, in the reference the set of admissible velocities is the same for every trajectory of the sequence, while in our case it changes for every $j$.
    However, this fact does not affect the argument thanks to the assumption $X_I^j \to X_I^\infty$ uniformly as $j\to+\infty$.

    Thus, denoting by $u_\infty$ the control corresponding to $q_\infty$, again by Theorem 3.41 in \cite{ABB} we have
    \begin{equation}
        \|u_\infty\|_1 \leq \liminf_{j\to+\infty} \|u_j\|_1
    \end{equation}
    Hence, we obtain
    \begin{equation}
        \MCeps(\Gamma,T,X_I^\infty)
        \leq
        \frac{1}{\eps}
        \|u_\infty\|
        \leq
        \frac{1}{\eps}
        \liminf_{j\to+\infty} \|u_j\|_1
        \leq
        \liminf_{j\to+\infty}
        \MC_\varepsilon (\Gamma,T,X_I^j) + \delta,
    \end{equation}
    which, thanks to the arbitrariness of $\delta$, yields the desired inequality.
\end{proof}

\subsection{Decomposition of the drift and proof of \eqref{eq:main-thm2}}
In this Subsection, we first prove that, if we replace the drift $X_0$ with its vertical part $X_0^\Gamma$, the motion complexity changes by a factor of order $O(\eps^{-1})$.
This allow us to prove directly \eqref{eq:main-thm2} of Theorem \eqref{thm:main}
\begin{prop}
    \label{prop:drift-vert}
    Let $X_0 = X_0 ^\Delta + X_0 ^\Gamma$, where $X_0^\Delta$ and $X_0^\Gamma$ are vector fields defined in a neighbourhood of $\Gamma$ that satisfy \eqref{eq:drift-decomposition}.
    Then,
    \begin{multline}
        \MCeps(\Gamma,T,X_0 ^\Gamma)
        =
        \frac1\varepsilon
        \inf\bigg\{
        J(u)
        \mid
        u \in L^\infty\left( [0,T];\mathbb{R}^m \right),
        \dot q_u = X_0 ^\Gamma (q_u) + \sum_{j=1}^m u_j X_j(q),
        \\
        q_u([0,T]) \subset \operatorname{Tube}_{\varepsilon}(\Gamma),
        \,
        q_u(0)=q_1,
        \,
        q_u(T)=q_2
        \bigg\}.
    \end{multline}
    Then, there is a constants $0 \leq C<+\infty$ and $\eps_0>0$ such that for $0<\eps<\eps_0$ it holds that
    \begin{equation}
        |\MCeps(\Gamma,T,X_0) - \MCeps(\Gamma,T,X_0^\Gamma) |
        \leq
        C\varepsilon^{-1}.
    \end{equation}
\end{prop}
\begin{proof}
    Let $v$ be an admissible control for the problem $\MCeps(\Gamma, T, X_0 ^\Gamma)$. Then
    \begin{equation}\label{eq:dot-q-gamma-X0}
        \dot q_v = X_0 ^\Gamma(q_v) + \sum _{j=1}^m v_j(t)X_j(q_v)
    \end{equation}
    Since, by definition, $X_0^\Delta(q)= \sum_{j=1}^m a_j(q) X_j(q)$, $q\in M$, for some suitable choice of $a=(a_1,\dots,a_m)\in C^\infty(M)^m$, we have that $v+a\circ q_v$ is an admissible control for $\MCeps(\Gamma, T, X_0 )$ corresponding to the same trajectory $q_v$.
    Notice that we can choose $a$ such that $|a|=|X_0^\Delta|$, where $|X_0^\Delta|$ is the sub-Riemannian norm introduced in \eqref{eq:metric}.
    So, for any admissible control $v$ for $\MCeps(\Gamma,X_0^\Gamma)$ there is a control $u$ admissible for $\MCeps(\Gamma,X_0)$ with cost
    \begin{equation}
        \label{ineq:drift-leq-vert-drift}
        J(u) = \int_0 ^T |u(t)| dt \leq J(v)+ \int_0 ^T |a(q_v(t))|dt
        \leq J(v) + C_1,
    \end{equation}
    where $0<C_1 \leq T\max_{q\in \operatorname{SR-Tube}_\eps(\Gamma)} |a(q)|$.
    Similarly, for any control $u$ admissible for $\MCeps (\Gamma,T,X_0)$, defining  $ v = u-a\circ q_u$ we obtain an admissible control for the problem $\MCeps (\Gamma,T,X_0 ^\Gamma)$ corresponding to the same trajectory $q_u$ and
    \begin{equation}
        \label{ineq:drift-leq-vert-drift2}
        J(v) = \int_0 ^T |v(t)| dt \leq J(u)+ \int_0 ^T |a(q_u(t))|dt
        \leq J(u) + C_2,
    \end{equation}
    where again $0<C_2\leq T\max_{q\in \operatorname{SR-Tube}_\eps(\Gamma)} |a(q)|$.
    So, we can conclude that there is some constant $C>0$ for which it holds the following:
    for all $u$ admissible for $\MCeps(\Gamma,T,X_0)$ (resp. $v$ admissible for $\MCeps(\Gamma,T,X_0^\Gamma)$) there is $v$ admissible for $\MCeps(\Gamma,T,X_0^\Gamma)$ (resp. $u$ admissible for $\MCeps(\Gamma,T,X_0)$) such that
    \begin{equation}
        |J(u)-J(v)|\leq C.
    \end{equation}
    Now, fix $\eps>0$ and take a minimizing sequence $(v_k)_{k\in\N}$ for $\MCeps(\Gamma,T,X_0 ^\Gamma)$, that is
    \begin{equation}
        \MCeps(\Gamma,T,X_0^\Gamma) = \lim _{k\to+\infty}\frac{1}{\eps}J(v_k),
    \end{equation}
    and let $u_k = v_k+a_k\circ q_{u_k}$, as above.
    We obtain
    \begin{equation}
        \frac{1}{\eps} J(v_k) + \frac{C}{\eps}
        \geq
        \frac{1}{\eps}J(u_k)
        \geq
        \MCeps(\Gamma,T,X_0) ,
    \end{equation}
    and passing to the limit for $k\to +\infty$, we obtain
    \begin{equation}
        \MCeps(\Gamma,X_0^\Gamma) + \frac{C}\eps \geq \MCeps(\Gamma,X_0 ).
    \end{equation}
    Exchanging $X_0$ with $X_0 ^\Gamma$, the very same argument yields also
    \begin{equation}
        \MCeps(\Gamma,X_0 )  + \frac{C}\eps \geq \MCeps(\Gamma,X_0 ^\Gamma).
    \end{equation}
    This completely proves the statement.
\end{proof}

Now, we can prove directly the asymptotic for $T=T_\Gamma$ in Theorem \ref{thm:main}.

\begin{prop}
    \label{cor:T=TGamma}
    Suppose that $T=T_\Gamma$. Then \eqref{eq:main-thm2} in Theorem \ref{thm:main} holds.
\end{prop}

\begin{proof}
    If $T=T_\Gamma$, then the curve $\Gamma$ is an admissible trajectory of the control system
    \begin{equation}
        \dot q_u = X_0 ^\Gamma (q_u) + \sum_{j=1}^m u_j X_j(q_u),
    \end{equation}
    with control $v_0=0$, which is of course the unique minimizer for $\MCeps(\Gamma,T,X_0 ^\Gamma)$ and in particular $\MCeps(\Gamma,T,X_0 ^\Gamma)=0$, for every $\eps\geq0$.
    Moreover, we know that $\Gamma$ is an admissible trajectory also for $\MCeps(\Gamma,T,X_0)$, hence there is some constant $C>0$ such that
    \begin{equation}
        \MCeps(\Gamma,T,X_0 ) \simeq C\eps^{-1}, \quad \eps \to 0.
    \end{equation}
    In particular, the control $u_0(t)=-a(e^{t{X_0^\Gamma}}(q_1))$ is admissible for $\MCeps(\Gamma,T,X_0)$ for every $\eps \geq 0$ and realizes the trajectory $\Gamma$.
    Hence, we have
    \begin{equation}
        \lim _{\eps \to 0} \eps\MCeps(\Gamma,T,X_0)
        =
        J(u_0).
    \end{equation}
    If we choose any form $\omega$ associated with $(\Delta,\Gamma)$, then, using the parametrization $\Gamma(t)=e^{t{X_0^\Gamma}}(q_1)$, we give the following intrinsic formula for $J(u_0)$:
    \begin{equation}
        \int_0^T |u_0(t)| dt
        =
        \int_0 ^T |a(e^{t{X_0^\Gamma}}(q_1))| dt
        =
        \int_0 ^T \frac{|X_0 ^\Delta (\Gamma(t))|} {\omega_{\Gamma(t)}\big(X_0(\Gamma(t))\big)} \omega_{\Gamma(t)}(\dot \Gamma(t))\,dt
        =
        \int_\Gamma \frac{|X_0 ^\Delta|} {\omega(X_0)} \omega,
    \end{equation}
    where the second equality follows since $\omega(\dot \Gamma)_{|\Gamma(t)}=\omega(X_0)_{|\Gamma(t)}$.
    Hence, we obtain \eqref{eq:main-thm2} in Theorem \ref{thm:main}.
\end{proof}

\subsection{Normal forms}
\label{subsec:norm-coord}
Here and in the following, we let $\omega$ be a form associated with $(\Delta,\Gamma)$. By requiring that $\omega(\dot\Gamma)\equiv 1$, this singles out an associated parametrisation $\Gamma:[0,S]\to M$, that we will exploit to construct normal coordinates in the tube $\operatorname{Tube}_\varepsilon(\Gamma)$.

In view of Corollary \ref{cor:T=TGamma}, from now on we always assume $T\neq T_\Gamma$.
Since we are going to show that $\operatorname{MC}_\eps (\Gamma,T) = O(\eps^{-2})$, we can neglect every term in the asymptotic of smaller order.
Thanks to this fact, we can introduce some normal forms both for the drift and the distribution.

Having fixed $\omega$ associated with $(\Delta, \Gamma)$, at each $q\in M$ it is possible to define, via the metric $g$, the skew-symmetric endomorphism $A_q:\Delta(q) \to \Delta(q)$ as follows:
\begin{equation}
    \label{eq:def-Aq}
    g_q(A_q v,w)=d\omega(v,w)\qquad \forall v,w\in \Delta(q).
\end{equation}
As such, $A_q$ has purely imaginary spectrum.
We denote its non-trivially zero eigenvalues by $\pm i\alpha_1(q),\ldots, \pm i\alpha_{k}(q)$, $k=n/2$ if $n$ is even or $k=(n-1)/2$ otherwise. 
Notice that, since we are assuming that $\Delta$ is a step 2 distribution, for every $q\in \Gamma$ there are $v,w\in \Delta(q)$ such that $d\omega(v,w)\neq 0$. 
In particular, $A_q$ cannot be the zero endomorphism, hence it has at least a non-zero eigenvalue.
By Lemma \ref{lem:continuity-MC}, it is not restrictive to suppose that the greatest eigenvalue in absolute value is simple. 
That is, we can order the eigenvalues of $A_q$ in such a way that
\begin{equation}
    \alpha_1(q)\ge \max\{\alpha_2(q),\ldots, \alpha_k(q)\}\ge 0.
\end{equation}
Notice that $\al_1 = \al$, where $\al$ was defined in \eqref{eq:alpha}.
Then, there exists an orthonormal frame $X_1,\ldots, X_{n-1}$ such that $A_q$ is block diagonal with $k$ two-dimensional blocks of the form
\begin{equation}
    \begin{pmatrix}
        0 & \alpha_i(q) \\ -\alpha_i (q) & 0
    \end{pmatrix}.
\end{equation}

Let $\mathfrak X$ denote the set of those $(\Delta,g,\Gamma)$ satisfying \textbf{(H0)}, \textbf{(H1)}, \textbf{(H2)}. This is the set of relevant motion planning problems, in the language of \cite{c1}.
Then, an easy modification of \cite[Theorem~2.2]{c1} shows that there exists an open and dense subset $\mathfrak X_*\subset \mathfrak X$ such that for any $(\Delta,g,\Gamma)\in \mathfrak X_*$  it holds \textbf{(H3)} if $n\ge 4$ and $(\Delta,\Gamma)$ is of generic type according to Definition~\ref{def:generic-type} if $n=3$.

We have the following, which is can be proved as \cite[Theorem~3.4]{c1} and references therein.

\begin{thm}[Normal coordinates]
    \label{thm:normal-coord}
    Let $(\Delta, g, \Gamma)\in \mathfrak{X}_*$.
    Then, there exists $\varepsilon_0>0$ and a coordinate system $(\xi, z):\mathbb{R}^{n-1}\times \mathbb{R}\to \operatorname{Tube}_{\varepsilon_0}(\Gamma)$ such that
    \begin{itemize}
        \item $\Gamma(s) = (0,s)$ for $t\in [0,S]$. In particular, $q_1=(0,0)$ and $q_2=(0,S)$.
        \item There exists a generating frame $\{X_1,\ldots,X_{n-1}\}$ for $\Delta$, which is orthonormal for the metric $g$, and that reads
              \begin{equation}
                  X_i(\xi,z) = \sum_{j=1}^{n-1} Q_{ij}(\xi,z)\partial_{\xi_j} + L_j(\xi,z)\partial_z,
              \end{equation}
              where the matrix $Q(\xi,z)$ is symmetric, $Q(\xi,z)\xi=0$, $L(\xi,z)\cdot \xi = 0$, and
              \begin{equation}
                  \label{eq:QL}
                  Q(\xi,z) = \operatorname{Id} + O(|\xi|^2)
                  \qquad\text{and}\qquad
                  L(\xi,z) = A_{(0,z)}\xi + O(|\xi|^2).
              \end{equation}
        \item If $n=3$, letting $\xi=(x,y)\in \mathbb{R}^2$, the above orthonormal frame reads
              \begin{gather}
                  X_1(x,y,z) = \left( 1 + y^2 \beta(x,y,z)\right)\partial_x - xy \beta(x,y,z) \partial_y + \frac{y}{2}\gamma(x,y,z) \partial_z,\\
                  X_2(x,y,z) = -xy \beta(x,y,z) \partial_x + \left(1+x^2\beta(x,y,z) \right)\partial_y - \frac{x}{2}\gamma(x,y,z)\partial_z,
              \end{gather}
              where $\beta$ and $\gamma$ are smooth functions. Moreover, $|\gamma(0,0,z)| = |\omega([X_1,X_2])|_{(0,0,z)}|= \alpha(0,0,z)$, where $\alpha$ is defined in \eqref{eq:alpha}.
    \end{itemize}
\end{thm}

\begin{rem}
    In the case $n\ge 4$, \cite[Theorem~3.4]{c1} holds under a stronger assumption than \textbf{(H3)}. Namely, the authors assume that all non-zero eigenvalues of $A_q$ are simple, and that $\dim\ker A_q$ is the smallest possible while retaining genericity.
    This requirement is due to the fact that in \cite{c1} the authors require a precise ordering of the blocks in $A_{(0,z)}$, while we only need to identify the block corresponding to the eigenvalue having largest modulus.
\end{rem}

\section{Step 2 case: the reduced metric complexity}
\label{sec:rmc-equiv}
In this Section we collect a series of technical results needed for the proof of the step 2 case.
Throughout this Section, we assume that $\Delta$ is a distribution on $M$ satisfying the assumptions of Theorem \ref{thm:main}.

First, we introduce the notion of \emph{reduced motion complexity} and we show that the problem of finding an asymptotic as $\eps \to 0$ for $\MCeps(\Gamma,T)$ can be reduced to finding an asymptotic of this reduced motion complexity (see Theorem \ref{thm:reduction} and Subsection \ref{subsec:proof-of-red}).
Then, in Subsection \ref{subsec:prop-of-rmc}, we prove some important properties of the reduced metric complexity that will be used in the proof of Theorem \ref{thm:main}.

We set $a_0 \coloneqq \omega(X_0) \in C^\infty(M)$.
Also, with a little abuse of notation, for $z\in\R$, we write $\al(z)\coloneqq \al(0,z)$ and $a_0(z)=a_0(0,z)$,  where in the r.h.s. we are using the coordinates $(\xi,z)$ introduced in Theorem \ref{thm:normal-coord}.
\begin{defn}
    Given $z_f>0$, for any $\varepsilon>0$ and $T>0$, the \emph{reduced metric complexity} is
    \begin{equation}
        \label{eq:def-rmc}
        \rmc_\varepsilon(z_f,T) = \inf \frac{2}{\varepsilon^2}J(v), \qquad J(v) = \int_0^T|v|\,dt,
    \end{equation}
    where the infimum is taken over all controls $v\in L^\infty([0,T];\mathbb{R})$ such that there exists a solution $z:[0,T]\to \mathbb{R}$ of
    \begin{equation}
        \label{eq:prob-red}
        \dot z = a_0(z)-\alpha(z)v,
        \qquad
        \text{with}
        \qquad z(0)=0\text{ and } z(T)=z_f.
    \end{equation}
\end{defn}

\begin{rem}
    \label{rem:free-syst}
    Since under the \textbf{(H3)} assumption we have $\al>0$, the control system \eqref{eq:prob-red} is feedback equivalent to the free system
    \begin{equation}
        \dot z = \tilde v(t),
        \qquad
        \tilde v : [0,T] \to \R.
    \end{equation}
    Indeed, for any fixed value of $z$, the transformation $ v \mapsto a_0(z) - \frac{r}{2}\al (z) v $ is a diffeomorphism from $\R$ to $\R$.
    Hence any function in $z \in W^{1,\infty}([0,T])$ is an admissible trajectory.
\end{rem}

The main two results of this Section are the following.

\begin{thm}
    \label{thm:reduction}
    Let $T\neq T_\Gamma$. Then,
    \begin{equation}
        \MCeps(\Gamma,T)
        =
        {\rmc}_\varepsilon\left(\int_\Gamma \omega,T \right) + O(\varepsilon^{-1}),
        \quad
        \text{as }
        \eps \to 0
        .
    \end{equation}
\end{thm}

As a consequence of the above, if we can find $C>0$ such that
\begin{equation}
    \label{eq:asymp-rmc-spiegazione}
    {\rmc}_\varepsilon\left(\int_\Gamma \omega,T \right)
    \simeq
    \frac{C}{\eps^2} , \quad \text{ as } \eps \to 0.
\end{equation}
then, we have, for the same constant $C>0$,
\begin{equation}
    \MCeps(\Gamma,T)
    \simeq
    \frac{C}{\eps^2} , \quad \text{ as } \eps \to 0.
\end{equation}
Hence, obtaining the correct constant $C$ in  \eqref{eq:asymp-rmc-spiegazione} is enough to prove Theorem \ref{thm:main}.
This is done in Section~\ref{sec:proof-step-2}, and requires the following result, proved in Section~\ref{subsec:prop-of-rmc}.

\begin{thm}
    \label{thm:sign-of-control}
    The infimum in the definition of $\rmc_\varepsilon(z_f,T)$ can be taken on controls $v$ such that the corresponding $z$ coordinate is non-decreasing. Moreover, we can assume that $v\le 0$ if $T< T_\Gamma$ and that $v \ge 0$ if $T<T_\Gamma$.
\end{thm}

\subsection{Proof of Theorem \ref{thm:reduction}}
\label{subsec:proof-of-red}

The proof is composed of two steps:
first, we show how to use the normal coordinates introduced in Theorem \ref{thm:normal-coord} to rewrite the control system in a convenient way, which allow us to do precise computations.
Then, we show how we can reduce to consider just the vertical variable.

From now on we denote $d=\lfloor \frac{n-1}{2}\rfloor$.
From Theorem \ref{thm:normal-coord}, in the step 2 case, we have two alternatives.
If $\dim M$ is odd, then $n=2d+1$ and using the local coordinates $(\xi,z) = (x,y,z) \in \R^{d} \times \R^{d} \times \R$ introduced above, the vector fields of the control system \eqref{eq:cs-drift} can be chosen in the form
\begin{gather}
    \label{eq:norm-coord-explicit1}
    X_{2i-1} (x,y,z) = \partial_{x_i} + \alpha_i(z)\frac{y_i}2\partial_z + O(|\xi|^2),\\
    \label{eq:norm-coord-explicit2}
    X_{2i}(x,y,z) = \partial_{y_i} - \alpha_i(z)\frac{x_i}2\partial_z + O(|\xi|^2),
\end{gather}
for $i=1,\ldots, d$.
If, instead, $n=\dim M$ is even, then $n=2d+2$ and the normal forms for the $X_i$ are the same as \eqref{eq:norm-coord-explicit1},\eqref{eq:norm-coord-explicit2} for $i=1,\ldots, d$, and
\begin{equation}
    X_{n-1} = \pa_{x_{n-1}} + O(|\xi|^2).
\end{equation}

In both cases, it is an elementary fact that one can introduce polar coordinates $(r_i,\theta_i)\in \R_{>0} \times S^1$ in every $(x_i,y_i)$-plane by letting$(x_i,y_i)=r_i(\cos\theta_i,\sin\theta_i)$.
Letting, $\mathbf{v}=(v_1,\dots,v_d)$ and $\mathbf{w}=(w_1,\dots,w_d)$, where
\begin{equation}
    \label{eq:transformation-control}
    \begin{pmatrix}
        w_{i} \\ v_{i}
    \end{pmatrix}
    =
    R_{\theta}
    \begin{pmatrix}
        u_{2i-1} \\ u_{2i}
    \end{pmatrix},
    \qquad
    \text{for } i=1,\dots,d,
\end{equation}
the control system \eqref{eq:cs-drift} reads
\begin{equation}
    \label{eq:cs-cylind-coord}
    \begin{cases}
        \dot r_i = w_i + O(|r|),                 \\[1em]
        \dot\theta_i = -\dfrac{v_i}{r_i}+O(|r|), \\[.5em]
        \dot z = a_0(r,\theta,z)-\sum_{i=1}^d\alpha_i(r,\theta,z)\dfrac{r_i}{2}v_i + O(|r|),
    \end{cases}
    \quad
    \text{ for } i=1,\dots, d,
\end{equation}
with the extra equation $\dot x_{n-1} = u_{n-1}(t) + O(|r|^2)$ if $\dim M$ is even. Here, $O(|r|)$ stands for a function $\phi$ such that $|\phi(r,\theta,z)| \leq C |r|$ for every $(\theta,z)\in \mathbb T^d\times \R$, provided that $r\in \R^d_{>0}$ is in a sufficiently small neighbourhood of $0$.
Notice that in the equation for $\dot z$ we are using the smoothness of $\omega(X_0)$ to deduce that  $\omega(X_0) (r,\theta,z) = a_0(z) + O(|r|)$.

We now show that, up to an error of order $O(\eps^{-1})$, we can reduce to consider the previous system of equations without the $O(|r|)$ terms in the equations for the $r$ and $z$.
Moreover, up to another $O(\eps^{-1})$ term, we can replace $a_0(r,\theta,z)$ and $\al_i(r,\theta,z)$ with $a_0(0,0,z)$ and $\al_i(0,0,z)$, respectively.
To make this precise, we introduce the following intermediate notion of reduced motion complexity.
\begin{defn}[auxiliary-reduced motion complexity]
    Consider the control system
    \begin{equation}
        \label{eq:red-cs-cylind}
        \begin{cases}
            \dot r = \mathbf{w}, \\
            \dot z = a_0(z)-\sum_{i=1}^d\alpha_i(z)\dfrac{r_i}{2}v_i.
        \end{cases}
    \end{equation}
    where $r_i\geq 0$, $ z \in \R$ and $\mathbf{v},\mathbf{w}\in L^\infty([0,T]; \R^d)$ are the controls.
    The \emph{auxiliary reduced motion complexity} is
    \begin{equation}
        \operatorname{a-rmc}_\eps(z_f,T) \coloneqq \inf \| (\mathbf{v},\mathbf{w}) \|_{L^1},
    \end{equation}
    where the infimum is taken over all controls $\mathbf{v},\mathbf{w}\in L^\infty([0,T]; \R^d)$ such that the corresponding solution of \eqref{eq:red-cs-cylind} satisfies the boundary conditions $r(0) = 0$, $z(0)=0$, $r(T) = 0$, $z(T)=z_f$, and is subject to the constrain $|r(t)|\leq \eps$.
\end{defn}

\begin{lem}
    \label{lem:reduction-to-no-remainders}
    We have the following asymptotic equivalence
    \begin{equation}
        \label{eq:equiv-MC-armc}
        \MCeps(\Gamma,T)
        =
        \operatorname{a-rmc}_\varepsilon\left(\int_\Gamma \omega,T \right) + O(\varepsilon^{-1}),
        \quad
        \text{as }
        \eps \to 0
        .
    \end{equation}
\end{lem}
\begin{proof}
    By smoothness of $a_0$ and $\alpha_i$, the equation for $\dot z$ in \ref{eq:cs-cylind-coord} reads
    \begin{align}
        \dot z
         & =
        a_0(z)-\sum_{i=1}^d  \alpha_i(z) \dfrac{r_i}{2}v_i + O(|r|)
        ,
    \end{align}
    where the $O(|r|)$ are for $|r|\to 0$, uniformly in $\theta\in \mathbb T^d$ and we have used, again, the notation $a_0(z)$ for the function $a_0$ restricted to $r=0$ and similarly for the $\al_i$.
    Let $\phi_{r_i},\phi_{\theta_i},\phi_z$ be the $O(|r|)$ terms in the equations of $r_i,\theta_i,z$ respectively, $i=1,\dots,d$.
    Replacing $w_i,v_i$ in \eqref{eq:cs-cylind-coord} by $\tilde w_i = w_i-\phi_{r_i}$ and $ \tilde v_i = v_i + \frac{2\phi_z}{m r_i \alpha_i}$, then there exists a constant $C>0$ such that $\|w_i-\tilde w_i\|_\infty \le C$ and $\|v_i-\tilde v_i\|_\infty \le C$.
    This yields
    \begin{equation}
        \frac{1}{\eps}
        |
        J(\widetilde{\mathbf{v}},\widetilde{\mathbf{w}})
        -
        J(\mathbf{v},\mathbf{w})
        |
        \leq
        \frac{2dCT}{\eps}.
    \end{equation}
    Then, using an argument similar to the one used in Proposition \ref{prop:drift-vert}, one obtains that this modification impacts the expansion of the metric complexity starting at order $\varepsilon^{-1}$, so the equivalence in \eqref{eq:equiv-MC-armc} follows.
    Finally, since the variables $\theta_i$ do not appear in the equations of $r$, $z$, nor in the boundary conditions, we can neglect them.
\end{proof}

From the structure of system \eqref{eq:red-cs-cylind}, it is quite clear how an optimal trajectory moving in the $z$ direction should look like:
if we want to maximize $\dot z$, we should use the control $v_1$, corresponding to the biggest eigenvalue $\al$, with $r_1=\eps$, the maximum value allowed for $r_1$, for most of the time.
Moreover, since with this choice the $r$ are constant for most of the time, we are effectively reducing to a one dimensional control system for $z$.

Now, we make this intuition precise.

\begin{prop}
    \label{prop:red-to-1D}
    Let $T\neq T_\Gamma$. Then, for any $z_f>0$ the following asymptotic equivalence holds:
    \begin{equation}
        \operatorname{a-rmc}_\eps\left(z_f ,T\right) \simeq \operatorname{rmc}_\eps \left( z_f ,T\right) + O(1)
        \quad
        \text{as } \eps \to 0.
    \end{equation}
\end{prop}
In the proof of Proposition \ref{prop:red-to-1D}, we need the following technical Lemma.
\begin{lem}
    \label{lemma:cont-MC}
    Fix $\eps>0$ and let $(\delta_k)_{k\in\N}$ be a sequence such that $\delta_k \to 0$ as $k\to +\infty.$ Then
    \begin{equation}
        \label{eq:cont-armc}
        \lim_{k\to +\infty}
        \operatorname{a-rmc}_\eps (\Gamma,T+\delta_k)
        =
        \operatorname{a-rmc}_\eps (\Gamma,T)
    \end{equation}
\end{lem}
\begin{proof}
    Without loss of generality, we can assume that $\delta_k \to 0+$ as $k\to +\infty$. Indeed, the case $\delta_k \to 0-$ is similar, and the general case can be reduced to these two cases.

    The proof proceeds by contradiction.
    Since the arguments are essentially the same, to avoid repetitions we only consider the case where, up to subsequences, it holds
    \begin{equation}
        \label{ineq:lim-rmc}
        \lim_{k\to +\infty}
        \armc_\eps (\Gamma,T+\delta_k)
        <
        \armc_\eps (\Gamma,T).
    \end{equation}

    By definition of $\armc_\eps (\Gamma,T+\delta_k)$, we have that for any $k$ there is a control $v_k$ which is admissible for $\armc(\Gamma,T+\delta_k)$, and such that
    \begin{equation}
        \frac{1}{\eps}\|v_k\|_1 \leq \armc_\eps (\Gamma,T+\delta_k) + \frac1k.
    \end{equation}
    Moreover, since $\delta_k \to 0$ and $q_{v_k}(T+\delta_k)=\Gamma(S)$, for any $\eta>0$ there exists $k$ and $\tau>0$ such that $|q_{v_k}(T-\tau)-\Gamma(S)| < \eta$.
    Here $|\cdot|$ denotes the euclidean distance.
    Then, an argument similar to the one used in the first part of Lemma~\ref{lem:continuity-MC} allows to modify $v_k$ on the interval $[T-\tau,T]$ to obtain a control $u_k$ such that $q_{u_k}(T)=\Gamma(S)$.
    Moreover, an estimate again similar to Equation \eqref{eq:ineq-cont-drift-controllo-tempo-finale}
    yields
    \begin{equation*}
        \armc_\eps(\Gamma,T)
        \le
        \frac1\varepsilon\|u_k\|_1
        \le
        \armc_\eps (\Gamma,T+\delta_k) + \frac1k
        +
        C \left(\e^{1/2} +\tau^{1/2}\right).
    \end{equation*}
    Since for $\e\to 0$ we can take $k\to +\infty$ and $\tau\to 0$, this contradicts \eqref{ineq:lim-rmc}.
\end{proof}

Now, we prove Proposition \ref{prop:red-to-1D}.

\begin{proof}
    Letting $u=2v/\varepsilon$ in the definition of $\rmc$, we have
    \begin{equation}
        \label{eq:rmc2}
        \rmc_\varepsilon(z_f,T)= \inf \frac{1}\varepsilon\|u\|_1,
        \quad\text{where}\quad
        \dot z = a_0(z)-\alpha_1(z)\frac\varepsilon2u,
    \end{equation}
    We remark that this dynamic coincides with that of the $z$ variable in the dynamics \eqref{eq:cs-cylind-coord} of $\armc$, if $r=\varepsilon e_1$, where $e_1=(1,0,\ldots,0)$.

    We start by showing that
    \begin{equation}
        \label{eq:armc-geq}
        \armc_\varepsilon(z_f,T) \ge \rmc_\eps(z_f,T).
    \end{equation}
    Fix a control $(\mathbf v, \mathbf w)$ that is admissible for $\armc_\varepsilon(z_f,T)$.
    Considering the corresponding $z$ coordinates, given by \eqref{eq:cs-cylind-coord}, we then let
    \begin{equation}
        u = \sum_{i=1}^d \frac{r_i}{\varepsilon} \frac{\alpha_i(z)}{\alpha_1(z)}v_i, \qquad \text{on }[0,T].
    \end{equation}
    With this choice of control, the $z$ coordinate given by \eqref{eq:cs-cylind-coord} is the same appearing in the dynamic of \eqref{eq:rmc2}.
    Recall that $\alpha_1\ge \alpha_i$ for any $i=1,\ldots, d$, and that, by definition of $\armc_\varepsilon(z_f,T)$ we have $|r|=|(r_1,\ldots,r_d)|\le \varepsilon$.
    Thus, on $[0,T]$,
    \begin{equation}
        |u| \le \sum_{i=1}^d \left| \frac{r_i}\varepsilon\right| \left|v_i\right|
        \le \frac{|r|}{\varepsilon} |\mathbf v| \le |\mathbf v|.
    \end{equation}
    In particular, $\|u\|_1\le \|(\mathbf v,\mathbf w)\|_1$, completing the proof of \eqref{eq:armc-geq}.

    To complete the proof of the statement, we are left to show the opposite inequality. To this aim, we prove that there exists $C>0$, independent of $\varepsilon$, such that for any $\delta>0$ sufficiently small it holds
    \begin{equation}
        \label{eq:armc-leq}
        \armc_\varepsilon(z_f,T+\delta) \le \rmc(z_f,T)+C\left( 1 + \frac{\delta^{1/2}}\varepsilon\right).
    \end{equation}
    By Lemma~\ref{lemma:cont-MC}, taking the limit as $\delta\to 0$ yields that $\armc_\varepsilon(z_f,T) \le \rmc(z_f,T)+C$.

    Consider a control $u$ admissible for $\rmc_\varepsilon(z_f,T)$, in the formulation \eqref{eq:rmc2}. To prove \eqref{eq:armc-leq} it suffices to construct a control $(\mathbf v, \mathbf w)$ that is admissible for $\armc_\varepsilon(z_f,T+\delta)$ and such that $\|(\mathbf v, \mathbf w)\|_1\le \|w\|_1 + C(\varepsilon+\delta)$.
    The control $(\mathbf v, \mathbf w)$ can by obtained by concatenating three controls:
    \begin{enumerate}
        \item A control driving the system \eqref{eq:cs-cylind-coord} from $(0,0)$ to $(\varepsilon e_1, 0)$ on the time interval $[0,\delta/2]$;
        \item The control $\mathbf w = 0$ and $\mathbf v= (u(\cdot - \delta/2), 0, \ldots, 0)$ on the time interval $[\delta/2,T+\delta/2]$, which drives the system to $(\varepsilon e_1,z_f)$;
        \item A control driving the system to $(0,z_f)$ on the time interval $[T+\delta/2,T+\delta]$.
    \end{enumerate}
    Using the Ball-Box Theorem and a time rescaling as in the proof of Lemma~\ref{lemma:cont-MC} (see also \cite[Theorem~1.1]{prandi2013}), one proves that the cost for the first and the last step is bounded by $C(\varepsilon+\delta^{1/2})$, completing the proof of \eqref{eq:armc-leq} and thus of the statement.
\end{proof}
Finally, putting together Lemma \ref{lem:reduction-to-no-remainders} and Proposition \ref{prop:red-to-1D}, Theorem \ref{thm:reduction} follows.

\subsection{Proof of Theorem~\ref{thm:sign-of-control}}
\label{subsec:prop-of-rmc}
We split the proof of Theorem~\ref{thm:sign-of-control} in two parts: first we prove that $z$ is non-decreasing (Lemmas \ref{lemma: affine-approx} and \ref{lemma:dz-pos}), then we prove that the control $v$ has a sign, depending on the difference $T-T_\Gamma$ (Lemma \ref{lem:control_pos}).

\begin{lem}
    \label{lemma: affine-approx}
    Let $v$ be an admissible control for $\rmc_\varepsilon(z_f,T)$. Then, there exists $C>0$ such that for any $\eta>0$, there exists an admissible control $v_\eta$ such that $z_{v_\eta}$ is piecewise affine and
    $
        \big|
        \|v\|_{L^1} - \|v_\eta\|_{L^1}
        \big|
        \le C
            {\eta}.
    $
\end{lem}

\begin{proof}
    Thanks to Remark \ref{rem:free-syst},
    the set of admissible trajectories is $W^{1,\infty}([0,T])$, where affine functions are dense in the $W^{1,1}$ topology. So, for any admissible trajectory $z_v$ with control $v$, we can find another admissible trajectory $z_\eta$ with some other control $v_\eta$, such that $z_\eta$ is piecewise affine and
    \begin{equation}
        \label{eq:density-affine}
        \| z_v - z_\eta \|_{W^{1,1}} \leq  \eta
    \end{equation}

    Recall that $\dot z_v = a_0(z)-\al(z)v$. Hence, we can obtain a formula of $v$ in terms of $z$ and $\dot z$, which guarantees the existence of a constant $C>0$ such that
    \begin{equation*}
        \left|
        \int_0 ^T
        |v(t)| - |v_\eta(t)|
        dt
        \right|
        \leq
        \tilde C
        \bigg(
        \int_0 ^{T}
        |a_0 (z_v(t)) - a_0 (z_\eta(t))|
        dt
        +
        \int_0 ^{T}
        |\dot z_v(t) - \dot z_\eta(t)|
        dt
        \bigg)
        ,
    \end{equation*}
    The desired inequality follows by \eqref{eq:density-affine} and the fact that $a_0$ is Lipschitz.
\end{proof}

The following result proves the first part of Theorem~\ref{thm:sign-of-control}.

\begin{lem}
    \label{lemma:dz-pos}
    The infimum in \eqref{eq:def-rmc} can be taken on controls $v$ such that the corresponding $z_v$ is non-decreasing.
\end{lem}

\begin{proof}
    Let $v$ be a control admissible for \eqref{eq:def-rmc},\eqref{eq:prob-red}. The statement will be proved by showing that then there exists an admissible control $\tilde v$ such that $\dot z_{\tilde v} \ge 0$ and $\| \tilde v \|_1 \le \| v \|_1$.
    \\
    By Lemma \ref{lemma: affine-approx}, we can suppose that $z_v$ is piecewise affine:
    \begin{equation}
        z_v(t) =
        \sum_{j=1} ^k f_j (t) \mathds{1}_{[t_j,t_{j+1}]}(t),
        \quad
        f_j(t) = A_j t + B_j,
    \end{equation}
    where
    $
        A_j,B_j\in\R,
        \quad
        0<t_1<\dots<t_k<T
    $.
    We have to show that $A_j>0$ for every $j=1,\dots,k$.
    By contradiction, suppose that $A_{j_0} < 0$ for some $j_0\in\{1,\dots,k\}$.
    Then there is some $s_1>t_{j_0}$ such that $z_v(t_{j_0})=z_v(s_1) \eqqcolon z_1$.
    Then, the function $z_v$ attains a local minimum for some $t_{j_1} < s_1$, $j_1 > j_0$.
    Define $z_0 \coloneqq z_v(t_{j_1})$. Take any $s_0\in[0,t_{j_0})$ such that $z_v(s_0)=z_0$.
    Since the function $|a_0|/\al$ is continuous, it must have a minimum point $z_{\min}$ in the interval $[z_0,z_1]$
    Moreover, we can find $\tau_0\in[s_0,t_{j_0}]$ and $\tau_1\in[t_{j_1},s_1]$ such that $z_v(\tau_0)=z_v(\tau_1)=z_{\min}$.
    \begin{figure}[t]
        \centering
        \includegraphics[width=.5\textwidth]{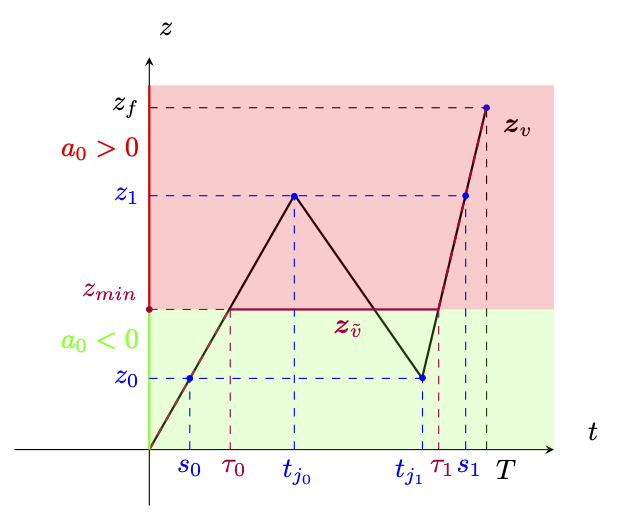}
        \caption{Construction for the proof of Lemma~\ref{lemma:dz-pos} in case $a_0(z_{\min})=0$.}
        \label{fig:dz_geq_0}
    \end{figure}

    As a consequence, there is some control $\tilde v$ such that $z_{\tilde v}$ coincides with $z_v$ in $[0,\tau_0]$ and $[\tau_1,T]$ and $z_{\tilde v}(t) = z_{\min}$ for $t\in(\tau_0,\tau_1)$. (See Figure~\ref{fig:dz_geq_0}.)
    Notice that $\tilde v(t)= 2a_0(z_{\min})/(\eps\al(z_{\min}))$ for $t\in(\tau_0,\tau_1)$.

    In order to complete the proof, we are left to show that $\|\tilde v \|_1 \le \|v\|_1$.
    By definition of $\tilde v$, we have that
    \begin{equation}
        \label{eq:tilde_v}
        \int_0 ^T | \tilde v(t) | dt
        =
        \int_0 ^{\tau_0} | v(t) | dt
        +
        \int_{\tau_0} ^ {\tau_1} \frac{| a_0(z_{min}) |}{ \al(z_{\min})} dt
        +
        \int_{\tau_1} ^T | v(t) | dt.
    \end{equation}
    Now, if $|a_0(z_{\min})| = 0$, then we have $ \|\tilde v \|_1 \le \|v\|_1$ trivially. Otherwise, on the interval $[\tau_0,\tau_1]$ we have either $a_0>0$ or $a_0<0$. In both cases, we have
    \begin{multline*}
        \int_{\tau_0} ^ {\tau_1} \frac{ | a_0(z_{min}) |}{ \al(z_{\min})} dt
        \leq
        \int_{\tau_0} ^ {\tau_1} \frac{  |a_0(z_v (t))| }{ \al(z_v (t))} dt
        =
        \pm
        \int_{\tau_0} ^ {\tau_1} \frac{ a_0(z_v (t)) }{ \al(z_v (t))} dt
        \\
        =
        \pm
        \int_{\tau_0} ^ {\tau_1}
        \left( \frac{ \dot z_v }{\al(z_v (t))}
        +
        v(t) \right)
        dt
        \leq
        \pm
        \int_{z(\tau_0)} ^ {z(\tau_1)}
        \frac{1}{ \al(z)} dz
        +
        \int_{\tau_0} ^ {\tau_1} |v(t)| dt
        =
        \int_{\tau_0} ^ {\tau_1} |v(t)| dt .
    \end{multline*}
    The last equality follows since $z_v(\tau_0)=z_v(\tau_1)$.
    Hence, we obtain that $\|\tilde v\|_1\le \|v\|_1$, thus completing the proof.
\end{proof}

We conclude the proof of Theorem ~\ref{thm:sign-of-control} thanks to the following result.

\begin{lem}
    \label{lem:control_pos}
    The infimum in \eqref{eq:def-rmc} can be taken on controls $v$ such that %
    \begin{enumerate}
        \item $v \le 0$ if $T < T_\Gamma$;
        \item $v \ge 0$ if $T > T_\Gamma $.
    \end{enumerate}
\end{lem}

\begin{proof}
    Assume that $ T > T_\Gamma $.
    In this case, the drift drives the system beyond the target.
    Namely, Letting $z_0$ be the solution of \eqref{eq:prob-red} with $v\equiv 0$, we have $z_0(T)>z_f$ and thus, for any control $v$ admissible for $\rmc_\varepsilon(z_f,T)$ it holds $v(t)>0$ for some $t\in[0,T]$.
    Let us show that for any control $v$ that is admissible for $\rmc_\varepsilon(z_f,T)$ and takes negative values, there exists a non-negative admissible control with smaller $L^1$ norm. By Lemma~\ref{lemma: affine-approx}, we can restrict to consider the case where $v$ is affine.
    Letting
    $
        t_*:=\inf \{t\in[0,T] \, \mid \, v(t) < 0 \}\in[0,T],
    $
    define
    \begin{equation*}
        \tilde v(t)
        =
        \begin{cases}
            v(t) \quad & \text{ if } t\in [0,t_*],
            \\
            0 \quad    & \text{ if } t\in [t_*,T].
        \end{cases}
    \end{equation*}
    Observe that if $v(0)=0$, then $t_*=0$, and hence $\tilde v(t)=0$ for all $t\in [0,T]$.

    There are three possibilities:
    \begin{enumerate}
        \item $z_{\tilde v}(T) < z_f$;
        \item $z_{\tilde v}(T) = z_f$;
        \item $z_{\tilde v}(T) > z_f$.
    \end{enumerate}
    Clearly, if $z_{\tilde v}(T) = z_f$, then  $\tilde v$ is the control we were looking for.

    Assume that we are in the first case. Let us consider the function $\phi(t):= z_0(T-t;z_v(t)) - z_f$, where $z_0(s;\zeta)$ denotes the solution of the Cauchy problem $\dot z = a_0(z)$, $z(0)=\zeta$, evaluated at time $s$.
    By the continuous dependence on the initial value of a solution of a differential equation, we have that $\phi$ is continuous.
    Moreover, it is positive at $t=0$, and we have
    $\phi(t_*)=z_{\tilde v}(T) - z_f < 0$.
    So, we there exists a time $t_1\in[0,t_*]$ such that the control equal to $v$ up to time $t_1$ and $0$ for $t \ge t_1$ steers the system from $0$ to $z_f$.
    By construction, this control is never negative and its $L^1$ norm is smaller
    or equal than the norm of $v$.

    In the third case it holds $z_{\tilde v}(T)>z_v(T)$. Moreover,  $T>T_\Gamma$ implies $a_0>0$ on the whole curve $\Gamma$ and hence that $\dot z_v(t_*+) > \dot z_{\tilde v}(t_*+)$. Then, there exists $t_1\in [t_*,T]$ such that $z_v(t_2)=z_{\tilde v}(t_2)$. This implies that the control
    \begin{equation*}
        \vardbtilde v(t)
        =
        \begin{cases}
            v(t) \quad \text{ if } t\in [0,t_*]\cup[t_2,T],
            \\
            0 \quad \text{ otherwise},
        \end{cases}
    \end{equation*}
    steers the system from $0$ to $z_f$ and $\|\vardbtilde v\|_1\leq\|v\|_1$. This ends the case $ T > T_\Gamma $.

    \begin{figure}
        \begin{subfigure}{0.4\textwidth}
            \input{fig2}
            \caption{Case $z_0 (T)>z_f$ and $z_{\tilde v}(T)<z_f$.}
            \label{fig:lemma2}
        \end{subfigure}
        \hfill
        \begin{subfigure}{0.4\textwidth}
            \input{fig3}
            \caption{Case $z_0 (T)>z_f$ and $z_{\tilde v}(T)>z_f$.}
            \label{fig:3}
        \end{subfigure}
        \begin{subfigure}{0.4\textwidth}
            \input{fig4}
            \caption{Case $z_0 (T)<z_f$ and $z_{\tilde v}(T)<z_f$.}
            \label{fig:4}
        \end{subfigure}
        \hfill
        \begin{subfigure}{0.4\textwidth}
            \input{fig5}
            \caption{Case $z_0 (T)<z_f$ and $z_{\tilde v}(T)>z_f$.}
            \label{fig:5}
        \end{subfigure}
        \caption{Construction for Lemma \ref{lem:control_pos}.}
        \label{fig:lemma-v-has-sign}
    \end{figure}
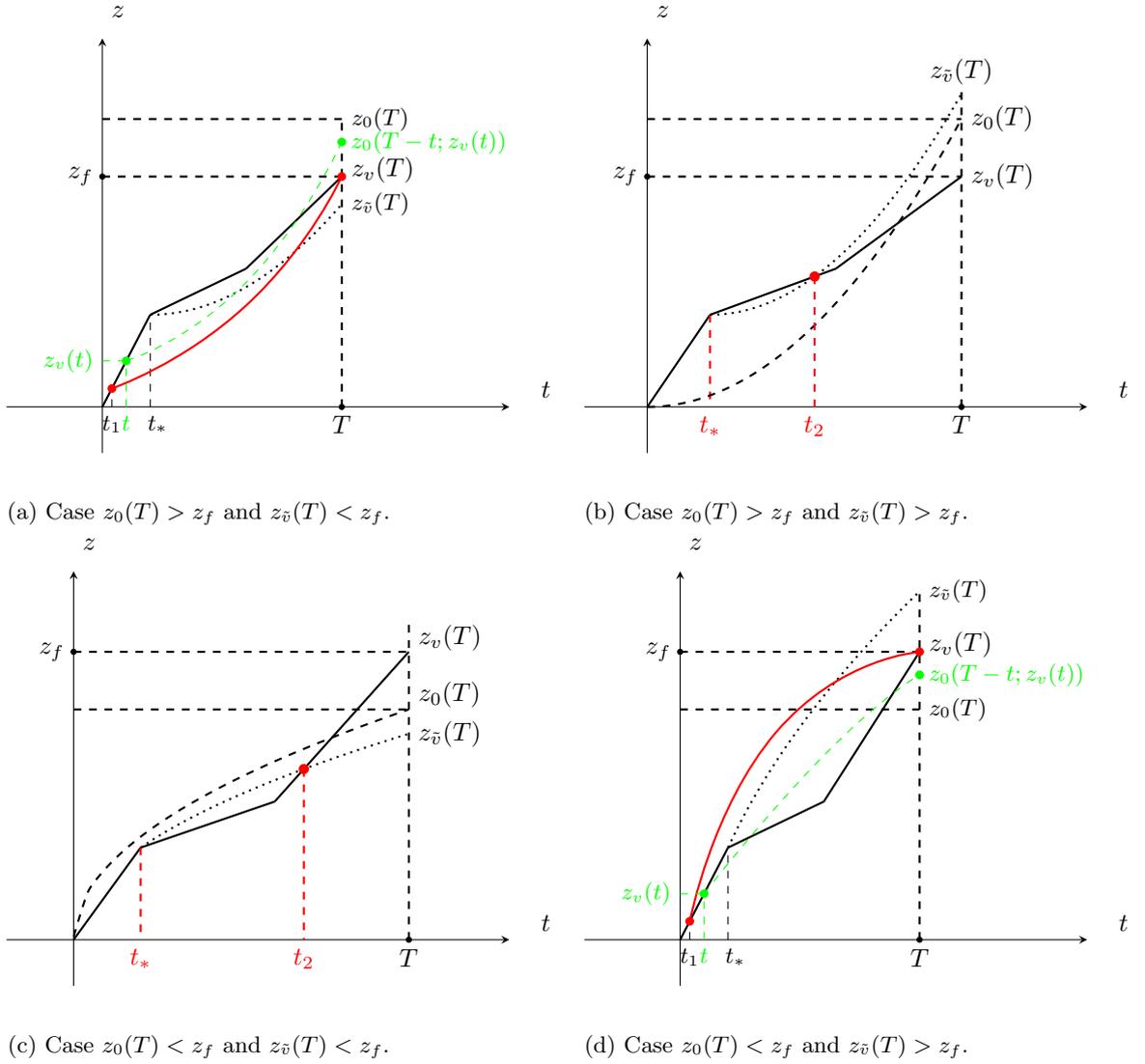

    Assume now $ T < T_\Gamma $.
    By Lemma~\ref{lemma:dz-pos} it holds that $0<\dot z_v(t) = a_0(z(t))-\alpha(z(t))v(t)$, which implies that $v(t) < a_0(z(t))/\alpha(z(t))$.
    Hence, $v(t)\le 0$ for any $t$ such that $a_0(z(t))\le 0$.
    This allows us to restrict to the case of $a_0>0$ on the whole $\Gamma$.
    In this case, one can adapt the argument for the case $T>T_\Gamma$ to show that $v\le 0$ (see also Figure~\ref{fig:lemma-v-has-sign}).
    We have that $\phi(0)<0$, where $\phi$ is defined as before.
    In case $z_{\tilde v}(T) < z_f$, we have that the trajectory $z_{\tilde v}$ must intersect the trajectory $z_0$ at some $t_2>t_*\coloneqq \{t\in[0,T] \mid v(t)>0\}$, so that one can argue as for point 3. in the case $T>T_\Gamma$.
    On the other hand, if $z_{\tilde v}(T) > z_f$, then we have that
    $\phi(t_*)>0$, so one can argue as for point 1. in the case $T>T_\Gamma$.
\end{proof}

\section{Step 2 case: proof of Theorem \ref{thm:main}}
\label{sec:proof-step-2}
In this section, we assume that $T\neq T_\Gamma$.
First, we define the following quantities
\begin{gather}
    \label{def:I^*}
    I^*(\Gamma,T) =
    \sup
    \bigg\{
    \int_0^T \frac{a_0(z(t))}{\alpha(z(t))}\,dt
    \mid
    z \text{ solution of \eqref{eq:prob-red} with } v\le 0
    \bigg\},\\
    \label{def:I_*}
    I_*(\Gamma,T)
    =
    \inf
    \bigg\{
    \int_0^T \frac{a_0(z(t))}{\alpha(z(t))}\,dt
    \mid
    z \text{ solution of \eqref{eq:prob-red} with } v\ge 0
    \bigg\}.
\end{gather}

\begin{lem}
    \label{lem:mc-I}
    Under the assumptions of Theorem~\ref{thm:main}, we have
    \begin{enumerate}
        \item If $ T < T_\Gamma $:
              \begin{equation}
                  \label{eqn:lower-bound-1}
                  \operatorname{MC}_\varepsilon(\Gamma,T)
                  \simeq
                  \frac{2}{\varepsilon^2}
                  \left[
                      \int_{0}^{S} \frac{dz}{\alpha(z)} - I^ *(\Gamma,T)
                      \right].
              \end{equation}
        \item If $T>T_\Gamma$:
              \begin{equation}
                  \label{eqn:lower-bound-2}
                  \operatorname{MC}_\varepsilon(\Gamma,T)
                  \simeq
                  \frac{2}{\eps^2}
                  \left[
                      I_*(\Gamma,T)-\int_{0}^{S} \frac{dz}{\alpha(z)}
                      \right].
              \end{equation}
    \end{enumerate}
\end{lem}

\begin{rem}
    It is possible to give a rather elementary and direct prove that the quantities in the squared brackets are indeed positive.
    For brevity of exposition, we will not prove it now. We are going to give a more explicit formula for them later, from which it will follow easily.
\end{rem}

\begin{proof}
    We only present an argument for the case $T<T_\Gamma$, since the other follows from an easy adaptation of the same argument.
    By Theorem~\ref{thm:reduction}, it suffices to prove the statement replacing $\MCeps(\Gamma,T)$ by $\operatorname{rmc}_\varepsilon(z_f,T)$, where $z_f=\int_\Gamma\omega$.
    Let $v$ be an admissible control for $\operatorname{rmc}_\varepsilon(z_f,T)$. By Lemma~\ref{lem:control_pos}, we can assume $v\le0$.
    By \eqref{eq:prob-red}, we can express $v$ as a function of $a_0$, $\alpha$.
    Then, for every such control, we have that
    \begin{equation}
        \label{eq:conto-costo}
        \|v\|_1
        =
        -\int_0^T
        v
        \,dt
        =
        \int_0^T
        \frac{\dot z - a_0(z)}{\alpha(z)}
        \,dt
        =
        \int_0^{T_\Gamma}
        \frac{dz}{\alpha(z)}
        -
        \int_0^T
        \frac{a_0(z)}{\alpha(z)}
        \,dt
        .
    \end{equation}
    Thus, taking the infimum w.r.t. non-positive controls that are admissible for $\operatorname{rmc}_\varepsilon(\Gamma,T)$, we obtain
    \begin{equation*}
        \operatorname{rmc}_\varepsilon(\Gamma,T)
        =
        \frac{2}{\varepsilon^2}\left[ \int_{0}^{T_{\Gamma}} \frac{dz}{\alpha(z)} -  \sup_{v\le 0}\int_0^T \frac{a_0(z)}{\alpha(z)}\,dt\right].
    \end{equation*}
    So, \eqref{eqn:lower-bound-1} follows by observing that the supremum on the r.h.s.~coincides with $I^*(\Gamma,T)$.
\end{proof}

In the following two lemmata, proven in Section~\ref{sec:proof-lemmata} we present closed formulae for $I^*(\Gamma,T)$ and $I_*(\Gamma,T)$.
Henceforth, for any $H\in\R$ and $c>0$, we denote by $z(\cdot;H,c)$ the solution of the ODE in \eqref{eq:prob-red}, with initial condition $z(0)=0$, and associated with the piece-wise constant feedback control given by
\begin{equation}
    \label{eq:controllo1}
    v(z;H,c)
    =
    \begin{cases}
        -c & \text{if } \: H>\dfrac{a_0(z)}{\al(z)}, \\[1em]
        0  & \text{if } \: H<\dfrac{a_0(z)}{\al(z)}.
    \end{cases}
\end{equation}

\begin{lem}
    \label{lem:I^*}
    If $T<T_\Gamma$, there exists a unique value of $H_\infty>0$ for which it holds that
    \begin{equation}
        \lim_{c \to +\infty} z(T;H_\infty,c) = z_f.
    \end{equation}
    Moreover, letting $\Omega = \{z\in[0,z_f] \mid a_0(z)/\alpha(z) > H_\infty\}$, we have
    \begin{equation}
        I^*(\Gamma,T) = \int_{\Omega} \frac{dz}{\alpha(z)}.
    \end{equation}
\end{lem}

\begin{lem}
    \label{lem:I_*}
    If $T>T_\Gamma$, it holds that
    \begin{equation}
        I_*(\Gamma,T) = \int_0^S \frac{dz}{\alpha(z)} + (T-T_\Gamma) \min_{\Gamma}\frac{a_0}{\alpha}.
    \end{equation}
\end{lem}

We are finally in a position to complete the proof of the theorem.

\begin{proof}
    The case $T=T_\Gamma$ follows from Proposition~\ref{cor:T=TGamma}.
    The case $T_\Gamma>T$ is an immediate consequence of Lemma~\ref{lem:mc-I} and Lemma~\ref{lem:I_*}, recalling that $a_0(z) = \omega(X_0)|_{(0,0,z)}$. The case  $T_\Gamma<T$ follows similarly, using Lemma~\ref{lem:I^*} and observing that
    \begin{equation}
        \int_0^S\frac{dz}{\alpha(z)} = \int_\Gamma \frac{1}{\alpha}\omega.
    \end{equation}
\end{proof}

\subsection{Proofs of the auxiliary lemmata}
\label{sec:proof-lemmata}

\paragraph{Proof of Lemma \ref{lem:I^*}}
First, observe that we have
\begin{multline}
    \label{eq:I^*+}
    I^*(\Gamma,T)
    =
    \sup
    \bigg\{
    \int_0^T \frac{a_0(z(t))}{\alpha(z(t))}\,dt
    \mid
    \dot z = a_0(z)+\alpha(z) v,\,
    z(0)=0,\, z(T)=z_f,\,
    v\ge 0
    \bigg\}.
\end{multline}
Here, we changed the sign of the control $v$ to reduce possible confusion.
For $c\in (0,+\infty]$, we let
\begin{equation}
    \label{def:Vc}
    V(c) := \inf\left\{\|v\|_1  \mid  0
    \le v\le c, \, z_v(T) = z_f\right\}
\end{equation}
Then, since $\dot z = a_0(z) + \al (z) v$, we have
\begin{equation}
    \int_0^T \frac{a_0(z(t))}{\al(z(t))} dt
    =
    \int_0^T \frac{\dot z(t)}{\al(z(t))} - v(t) \, dt
    =
    \int_0 ^{z_f} \frac{dz}{\al(z)} - \|v\|_{L^1},
\end{equation}
which implies
\begin{equation}
    \label{eq:Vinfty}
    I^*(\Gamma,T) = \int_0^{z_f}\frac{dz}{\alpha(z)} - V(+\infty).
\end{equation}
Hence, in the following we focus on computing $V(+\infty)$

First, we notice that we can apply the Pontryiagin Maximum Principle (PMP) to compute $V(c)$ for any $c<+\infty$. This will be enough thanks to the following result.
\begin{prop}
    \label{prop:bounded-controls}
    The function $V:(0,+\infty]\to \R_{\ge 0}$ is non-increasing. Moreover, $\lim_{c \to +\infty} V(c) =V(+\infty)$.
\end{prop}
\begin{proof}
    The monotonicity of $V$ is immediate from the definition. This implies the existence of the limit and the fact that $\lim_{c\to +\infty} V(c)\ge V(+\infty)$
    To see the opposite inequality, we consider a maximizing sequence $(v_j)_{j\in\N}$ for $V(+\infty)$.
    Since any $v_j$ is in $L^\infty$, then it is admissible for $V(\|v_j\|_\infty)$.
    This implies $V(\|v_j\|_\infty) \leq \|v_j\|_1$ for every $j\in\N$ and, in particular,
    $\lim_{j\to +\infty} V(\|v_j\|_\infty) \le \lim_{j\to +\infty} \|v_j\|_1 = V(+\infty)$.
\end{proof}

Now, we distinguish two cases: first, we consider the case of $\max_{z\in[0,z_f]} a_0(z) \geq 0$; then, we study the case $\max_{z\in[0,z_f]} a_0(z) < 0$.

Applying the PMP (see \cite{Agrachev2008}) to \eqref{def:Vc}, we obtain the following.
\begin{lem}
    \label{lemma:strutt-opt-contr}
    Assume that $\max_{z\in[0,z_f]} a_0(z) \geq 0$.
    Let $\tilde v$ be an optimal solution to \eqref{def:Vc} for $c$ sufficiently big. Then there exists $H \geq 0$ such that $\tilde v$ is the feedback control $\tilde v(t) = -v(z(t);H,c)$, where $v(z;H,c)$ is defined in \eqref{eq:controllo1}.
\end{lem}

\begin{proof}
    In order to apply the PMP, we define the Hamiltonian function
    \begin{equation}
        \label{def:hamilt-pmp}
        h(\xi,\xi_0,z,v)
        =
        \xi (a_0(z)+\alpha(z) v) + \xi_0 v,
        \quad
        u\in[0,c], \; \xi,z\in\R, \; \xi_0\in\{0,-1\}.
    \end{equation}
    Here, the choice $\xi_0=-1$ corresponds to normal extremals, while $\xi_0=0$ corresponds to abnormal extremals.
    By PMP, an extremal is a curve $(\xi(t),z(t))$, where $\xi$ satisfies the adjoint equation
    \begin{equation}
        \label{eq:hamilt-xi}
        \dot \xi
        =
        -\xi \big( a_0'(z) + \tilde v \al'(z) \big) .
    \end{equation}
    Here we denoted by $a_0'$ (resp.~$\alpha'$) the derivative of $\alpha$ (resp.~$a_0$) w.r.t.~$z$.
    Moreover, the value of the control $\tilde v$ maximizes the value of $h$, and hence
    the control  satisfies
    \begin{equation}
        \label{eq:extr-control}
        \tilde v(t)
        =
        \begin{cases}
            c & \text{ if } \; \xi \al(z)+\xi_0 > 0, \\
            0 & \text{ if } \; \xi \al(z)+\xi_0 < 0.
        \end{cases}
    \end{equation}
    If $\xi \al(z)+\xi_0 = 0$, $\tilde v$ is not determined directly, i.e., the control is singular.

    We start by showing that there are no abnormal extremals. In fact, if $\xi_0=0$,
    then \eqref{eq:extr-control} reads
    \begin{equation}
        \tilde v(t)
        =
        \begin{cases}
            c & \text{ if } \; \xi > 0, \\
            0 & \text{ if } \; \xi < 0.
        \end{cases}
    \end{equation}
    By the PMP, $\xi_0=0$  implies that $\xi\neq 0$ for all times.
    So, either $v\equiv0$ or $v\equiv c$.
    But, since $T\neq T_\Gamma$, up to considering $c$ sufficiently big, these controls do not steer the system from $0$ to $z_f$.

    We are then left with the normal case ($\xi_0=-1$). We now claim that there are no singular controls, which by \eqref{eq:extr-control} amounts to asking that $\xi\alpha(z)\neq 1$ a.e.~on $[0,T]$. We can differentiate the identity $\xi = 1/\alpha$ to obtain:
    \begin{equation}
        \label{eq:deriv-sing-cond}
        \dot \xi = -\frac{ \al'(z) \dot z }{\al(z)^2}
        = - \frac{ \al'(z) (a_0(z)+\alpha(z) \tilde v) }{\al(z)^2}
    \end{equation}
    Putting together \eqref{eq:deriv-sing-cond} and \eqref{eq:hamilt-xi}, and using again $\xi=1/\alpha$, yields
    \begin{equation}
        \left(\log \frac{a_0}\alpha \right)'=0
        \iff \left(\frac{a_0}{\alpha}\right)'= 0.
    \end{equation}
    By Lemma~\ref{lem:continuity-MC}, up to perturbing the drift $X_0$ we can assume that this identity cannot hold for $z$ in an interval $I\subset [0,z_1]$. So, this implies that the trajectory $z$ is constant. However, it is not difficult to show that under the hypothesis of \eqref{def:Vc}, trajectories for $z$ which are constant on some time interval are not optimal.

    From \eqref{eq:extr-control}, we have thus shown that the control $\tilde v$ stands defined by
    \begin{equation}
        \label{eq:controllo2}
        \tilde v(t) =
        \begin{cases}
            c & \text{ if } \: \xi(t) > \frac 1 {\al(z(t))}, \\
            0 & \text{ if } \: \xi(t) < \frac 1 {\al(z(t))}.
        \end{cases}
    \end{equation}
    Let $H$ be the constant value of the Hamiltonian $h$ along the extremal trajectory. Then, solving for $\xi$ the definition \eqref{def:hamilt-pmp} of $h$ we have
    \begin{equation}
        \xi(t)
        =
        \frac{H + \tilde v(t)}{a_0(z(t))+\al(z(t)) \tilde v(t)}.
    \end{equation}
    Since  $a_0$ and $\al$ are bounded, up to considering $c$ sufficiently big, we have  $a_0+\alpha c >0$.
    So, the statement follows from the fact that $\xi=1/\alpha(z)$ is equivalent to $H>a_0(z)/\alpha(z)$, which is easily checked using \eqref{eq:controllo2} and the definition of $H$.
\end{proof}

From Lemma \ref{lemma:strutt-opt-contr}, we have simply to find for which value of $H\in \R$ the control $\tilde v$ defined in \eqref{eq:controllo1} steers the system from $z(0)=0$ to $z(T)=z_f$.
\begin{claim}
    \label{claim:unique-extremal}
    Assume that $\max_{z\in[0,z_f]} a_0(z) \geq 0$.
    For any $c\in\R$ sufficiently large, there is a unique value $H(c)>0$ of $H$ such that the control given in \eqref{eq:controllo2} is the optimal control.
\end{claim}
For $H>0$, let $E_c(H) \coloneqq z(T;H,c)$ be the solution of $\dot z = a_0(z)+\al(z)v$, $z(0)=0$, evaluated at time $T$, where $v$ is defined as in \eqref{eq:controllo1}.
Then, $E_c$ is monotone non-decreasing. Indeed, if $H_1<H_2$, then $\dot z(t;H_1,c) \leq \dot z(t;H_2,c)$.
Moreover, since $T < T_\Gamma$,
we have that $\lim_{H\to 0+} E_c(H)<z_f$.
On the other hand, for any fixed $H>\max_{z\in[0,z_f]}a_0(z)/\al(z)$, we have that $z(T;H,c)\to +\infty$ if $c\to+\infty$.
So, if $c$ is large enough, we know that $E_c(H)>z_f$ for $H> \max_z a_0(z)/\al(z)$.
{Since $E_c$ is continuous}, this proves the Claim. \qed

\begin{claim}
    \label{claim:Hc-non-decr}
    The function $c\mapsto H(c)$ is monotone non-increasing. In particular, the following limit exists
    \begin{equation}
        \label{def:H-infty}
        H_\infty \coloneqq \lim _{c\to+\infty} H(c)
    \end{equation}
\end{claim}
Suppose, by contradiction, that there are $c_1<c_2$ such that $H(c_1)<H(c_2)$. Then, denoting by $z_1,z_2$ the trajectories of the control system corresponding to the resulting controls, we have that $\dot z_2 \geq \dot z_1$, with strict inequality on a set of positive measure. So, $z_2(T)>z_1(T)=z_f$, which is impossible since $z_2$ must satisfy the same boundary conditions of $z_1$.
This proves the Claim. \qed

Observe that $H_\infty$ can be also characterized by the property
\begin{equation}
    \label{eq:char-h-infty}
    \lim_{c\to +\infty} E_c(H_\infty) = \lim_{c\to +\infty} z(T;H_\infty,c) = z_f.
\end{equation}
Indeed, by the same argument above, the function $c\mapsto z(T;H,c)$ is strictly increasing for any $H>0$, so the limit in \eqref{eq:char-h-infty} exists and the limit function is again a monotone function of $H$. So, \eqref{eq:char-h-infty} uniquely determine the value of $H_\infty$.
\begin{figure}[ht]
    \centering
    \input{fig6}
    \caption{Graphical representation of the extremals controls. If the value of $a_0/\al$ is above $H$, then $v=0$ and $\dot z=a_0$. When the value of $a_0/\al$ is below $H$, then $v=c$. As $c$ grows, the extremal trajectories go faster and faster through the region of the $z$ axis where $a_0/\al < H$.}
    \label{fig:extr-controls}
\end{figure}
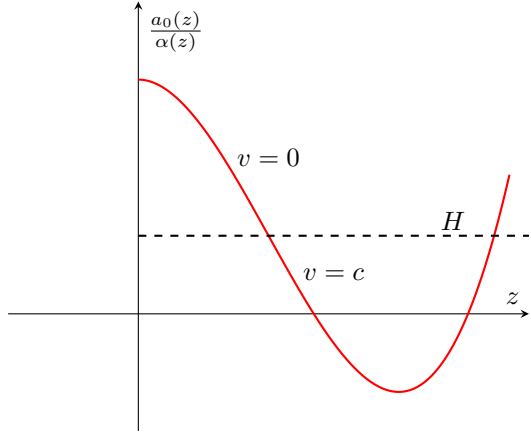

Concerning the limit trajectory $z_\infty$ as $c\to+\infty$, we see that it can be described as follows.
If $a_0(z(t))/\al(z) > H_\infty$ we simply apply zero control, that is $\dot z = a_0(z)$.
If at some point $z(t)$ we have $a_0(z(t))/\al(z(t)) = H_\infty$, the trajectory $z$ jumps to the next point $z_0\in(z(t),z_1]$ where equality $a_0(z_0)/\al(z_0) = H_\infty$ holds and we start again from this point with $\dot z = a_0(z)$.

It remains to compute the value of $I^*(\Gamma,T)$, which can be done exploiting \eqref{eq:Vinfty}.
To this aim, we let 
\begin{equation}
    \label{eq:def-Omega}
    \Omega(c) = \{ z\in[0,z_f] \, | \, a_0(z)/\al(z) > H(c)\}.
\end{equation}
Then, denoting by $z_c$ the optimal solution of $I^*(c)$, we have
\begin{align}
    I^*(c)
    &=
    \int_0 ^T \frac{a_0(z_c(t))}{\al(z_c(t))}dt
    =
    \int_{ z_c ^{-1}(\Omega(c)) } \frac{\dot z_c(t)}{\al(z_c(t))}dt
    +
    \int_{[0,T] \setminus z_c ^{-1}(\Omega(c)) } \frac{\dot z_c(t) - \al(z_c(t))c}{\al(z_c(t))} dt
    =
    \\
    &=
    \int_{\Omega(c)} \frac{dz}{\al(z)}
    +
    \int_{[0,z_f]\setminus \Omega(c)} \frac{dz}{\al(z)}
    -
    c\int_{[0,z_f]\setminus \Omega(c)} \frac{dz }{a_0(z) + \al(z )c}      
    \xrightarrow[c\to +\infty]{}
    \int_{\Omega} \frac{dz}{\al(z)}
    .
\end{align}
Hence, since $\lim_{c\to+\infty} I^*(c) = I^*$, this completes the proof of the statement of Lemma \ref{lem:I^*} for $\max_{z\in[0,z_f]} a_0(z) \geq 0$. 

Now, we have to consider the case $\max_{z\in[0,z_f]} a_0(z) < 0$. 
We have the following analogue of Lemma \ref{lemma:strutt-opt-contr} 
\begin{lem}
	Assume that $\max_{z\in[0,z_f]} a_0(z) < 0$. 
	Let $\tilde v$ be an optimal solution to \eqref{def:Vc}.
	 Then, for $c$ sufficiently big, there are $t_1,t_2\in[0,T]$, $t_1<t_2$ 
	 such that 
	 $\tilde v(t) = c$ for $t\in[0,t_1]\cup[t_2,T]$ and 
	 $\tilde v(t) = |a_0(z_{\min})|/\alpha(z_{\min})$ for $t\in(t_1,t_2)$, where $z_{\min} \coloneqq \arg \min |a_0(z)|/\al(z)$.  
\end{lem}
The proof of this Lemma follows closely the one of Lemma \ref{lemma:strutt-opt-contr}.
The difference here is that, since $a_0(z) < 0$ for every $z\in[0,z_f]$, for any $H>0$ we have that $v(t;H,c)=c$ for every $t\in[0,T]$. 
For such control it clearly holds that $z(T;H,c)\to +\infty$ as $c\to+\infty$, hence it is not admissible for our control problem. 
In order to arrive at the final point $z_f$ at time $T$, we can use a singular control that keep $z$ constant for the time needed, which corresponds to the time interval $(t_1,t_2)$ in the statement.  
In this case, $I_*$ is 
\begin{equation}
    I^*
    =
	T\frac{|a_0(z_{\min})|}{\al(z_{\min})}
	=
	T \min_{q\in \Gamma} \frac{|\omega(X_0)|}{\al(q)}.
\end{equation}
and hence
\begin{equation}
    \MCeps(\Gamma,T) 
    \simeq 
    \frac{2}{\eps ^2} 
    \left(
        \int_\Gamma \frac{1}{\al}\omega - T \min_{q\in \Gamma} \frac{|\omega(X_0)|}{\al(q)}
    \right) 
\end{equation}

\paragraph{Proof of Lemma \ref{lem:I_*}}
Similar computations as those leading to \eqref{eq:Vinfty} yield
\begin{equation}
    I_* (\Gamma,T) = \int_0^{z_f} \frac{dz}{\alpha(z)} + V(+\infty).
\end{equation}
We can then proceed as in the case of $I^*(\Gamma,T)$: we reduce to the optimal control problem with controls $v\in[0,c]$, then apply the PMP and take the limit as $c\to +\infty$.

In this case, we have the same extremals. However, since  by Lemma~\ref{lemma:dz-pos} we can restrict to consider trajectories with $\dot z \geq 0$, we see that for $c$ large enough it cannot happen that $v(t)= c$. So, the only possibilities are either $v=0$ or controls for which $z$ is constant. So, any extremal is the concatenation of these two categories of trajectories.

From assumption, we have that $ \min_z a_0(z)/\al(z) > 0$.
We know that if we simply apply the control $v_0 =0$ for every $t>0$ we get to the point $z_f$ in time $T_\Gamma<T$. So, in particular, with this control strategy, we arrive at a point $z_{\min}$ realizing the minimum of $a_0/\alpha$ at some time $\tau_0<T_\Gamma.$
Consider the control
\begin{equation}
    \label{eq:controllo-ott-I_*}
    v(t) =
    \begin{cases}
        0                                   & \text{if } t\in[0,\tau_0],                     \\
        \frac{a_0(z_{\min})}{\al(z_{\min})} & \text{if } t\in [\tau_0, \tau_0 + T-T_\Gamma], \\
        0                                   & \text{if } t \in [\tau_0 + T-T_\Gamma, T].
    \end{cases}
\end{equation}
The control clearly drives the system to the desired endpoint, and, since for $v=0$ we have $\dot z = a_0(z)$, we have
\begin{equation}
    \int_0^T \frac{a_0(z(t))}{\alpha(z(t))}\,dt
    =
    \int_\Gamma \frac{1}{\al}
    +
    (T-t_0) \frac{a_0(z_{\min})}{\al(z_{\min})}.
\end{equation}
Observe that $v$ is an extremal control. It is easily checked that any other extremal control has bigger or equal cost and thus that $v$ is the optimal control.

\section{Distributions of generic type}
\label{sec:Martinet}
The aim of this Section is to prove Theorem \ref{thm:main2}.
As discussed in the Introduction, see Remark \ref{rem:gen-distrib-dim-ge-4}, if $\dim M \geq 4$, then the assumption \textbf{(H3)} in Theorem \ref{thm:main} is generic.
If $\dim M=3$, this is no longer true.
For this reason, in this Section, we are going to deal with the case of $\dim M=3$ and where the couple $(\Delta,\Gamma)$ is of generic type, see Definition \ref{def:generic-type}.
\begin{rem}
    We stress that this assumption is again generic: if $\dim M=3$, the set $\mathcal X_3$ of couples $(\Delta,g,\Gamma)$ where $(\Delta,\Gamma)$ is of generic type is open and dense in the $C^\infty$ product topology among all couple of distributions of corank 1 and curves.
\end{rem}

We split the proof of Theorem \ref{thm:main2} in two parts.
First, in Subsection \ref{subsec:sing-martinet-point} we deal with the case of curve $\Gamma$ with a single Martinet point, see Theorem \ref{thm:single-martinet}.
After that, in Subsection \ref{subsec:many-Martinet-points}, we show how to reduce the general case to the first one, using some properties of the motion complexity.

To simplify the statements, throughout all the remaining part of this section, $\Delta$ is a corank 1 distribution on $M$, $\Gamma$ is a smooth simple curve transverse to $\Delta$, $\omega$ is a 1-form associated with $(\Delta,\Gamma)$, which we suppose to be a couple of generic type.
As before, we fix a parametrization of $\Gamma$ such that $\omega (\dot \Gamma) =1$.
Moreover, $g$ is a sub-Riemannian metric over $(M,\Delta)$ and the drift
$X_0$ is tangent to $\Gamma$.

\subsection{Case of a single Martinet point}
\label{subsec:sing-martinet-point}
In this section, we assume that on the curve $\Gamma$ there is only one Martinet point, which we denote by $\bar q$.

We start by recalling, from Theorem \ref{thm:normal-coord}, a convenient coordinate frame to carry out our computations.
First, thanks  to Theorem \ref{thm:normal-coord}, we fix coordinates $(x,y,z):\mathbb{R}^3\to \operatorname{Tube}_{\varepsilon_0}(\Gamma)$ such that $\bar q=(0,0,0)$,  $\Gamma(s) = (0,0,s)$, for $s\in[-S,S]$, and the expresssion of the drift $X_0$ and of an orthonormal frame $\{X_1,X_2\}$ for the metric $g$ have the following form:
\begin{align}
    X_0(x,y,z) & = a_0(x,y,z) \pa_z + O(|x|+|y|)
    \\
    \label{eq:norm-coord-martinet}
    X_1(x,y,z) & = \pa_x - \frac y 2 \gamma(x,y,z) \pa_z + O(x^2+y^2),
    \\
    X_2(x,y,z) & = \pa_y + \frac x 2 \gamma(x,y,z) \pa_z + O(x^2+y^2).
\end{align}
Here, $a_0,\gamma$ are smooth functions and $|\gamma(x,y,z)|=\alpha(x,y,z)$, where $\alpha=\left|\omega([X_1,X_2])\right|$ is the function defined in \eqref{eq:alpha}.

Observe that the function $\gamma$ in the above differs from the one appearing in Theorem \ref{thm:normal-coord} if $(x,y)\neq (0,0)$, but it follows from the following result that the two functions differ by a $O(x^2+y^2)$ term.
\begin{lem}
    \label{lem:Taylor-alpha}
    The function $\alpha=\left|\omega([X_1,X_2])\right|$ is Lipschitz continuous and smooth on the set $\{\alpha >0\}$.  Moreover, locally near $\bar q$, the set $\{\alpha>0\}$ has exactly $2$ connected components, $A_+$ and $A_-$, which are such that $(0,0,z)\in A_{+}$ if $z>0$ and $(0,0,z)\in A_{-}$ if $z<0$.
    Then, in coordinates we have
    \begin{equation}
        \label{eq:Taylor-alpha}
        \al(x,y,z)
        =
        \begin{cases}
            \kappa z + \be_1 x+\be_2 y + O(x^2+y^2+z^2),    & \text{on } A_+, \\
            -\kappa z - \be_1 x - \be_2 y + O(x^2+y^2+z^2), & \text{on } A_-,
        \end{cases}
    \end{equation}
    for some $(\be_1,\be_2)\neq (0,0)$ and $\kappa>0$ defined by \eqref{eq:def-k_i}.
\end{lem}
\begin{proof}
    The first part of the statement is a direct consequence of the definition of $\alpha$ and of the generic type assumption.
    Formula \eqref{eq:Taylor-alpha} follows directly from a Taylor expansion of $\alpha$ at the point $(0,0,0)$. Indeed, by Lemma~\ref{lemma:ciao}, we have that $\kappa = \lim_{z\to 0} |\partial_z\alpha(0,0,z)|$.
    The last requirement in the generic type assumption implies that $\kappa>0$, and the fact that $\alpha = \gamma$ on $A_+$ and $\alpha=-\gamma$ on $A_-$, implies that
    \begin{equation}
        \lim_{\substack{(x,y,z)\to 0\\(x,y,z)\in A_+}} \nabla \alpha(x,y,z) = \nabla \gamma(0,0,0)= -\lim_{\substack{(x,y,z)\to 0\\(x,y,z)\in A_-}} \nabla \alpha(x,y,z).
    \end{equation}
\end{proof}

Now, using the notation of Definition \ref{def:generic-type}, we have two possibilities:
either $\Sigma = \Sigma^+ = \{\bar q\}$ and $\Sigma^- = \emptyset$ or, vice-versa, $\Sigma = \Sigma^- = \{\bar q\}$ and $\Sigma^+ = \emptyset$.
In both cases, the definitions of $T_\Gamma$  that we gave before is still meaningful (see \eqref{eq:T-gamma}).
Hence, for the first case, we have the following result, which is an easy adaptation of Theorem \ref{thm:main}.
\begin{thm}
    \label{thm:single-Martinet-good-drift}
    Assume that $\Sigma = \Sigma^+ = \{\bar q\}$ and $\Sigma^- = \emptyset$.
    Then, the same asymptotic equivalences as in Theorem \ref{thm:main} hold. Moreover, in the case $T_\Gamma>T$, we have that $\bar q \in \operatorname{int} \Omega$, so the function ${1}/{\al}$ is bounded on $\Gamma\setminus\Omega$ and the integral in \eqref{eq:main-thm3} is finite.
\end{thm}

\begin{proof}
    One can prove the first two points proceeding exactly in the same way as in Theorem \ref{thm:main}.
    The only thing worth to notice is that, since $ \lim_{q\to \bar q} a_0(q)/\al(q) = +\infty $, clearly the minimum in \eqref{eq:main-thm1} is not attained at $\bar q$.

    Also the third point follows as in Theorem \ref{thm:main}, with minor adaptations.
    More precisely, the only fact one has to check is that the integral in \eqref{eq:main-thm3} is finite.
    To show this, we prove that the point $\bar q \in \operatorname{int} \Omega$, so it does not contribute to the value of the integral.
    To prove this, recall the definitions of $\Omega$ and the value $H_\infty$, see Equations \eqref{eq:def-Omega} and \eqref{def:H-infty}.
    From Claim \ref{claim:Hc-non-decr}, we know that $H_\infty \leq H(0) < +\infty$. Hence, $H_\infty$ must be finite.
    As a consequence, since $ \lim_{q\to \bar q} a(q)/\al(q) = +\infty $, the set $V \coloneqq \{ q \in \Gamma \mid  a(q)/\al(q) > 2H_\infty \}$ is an open neighbourhood of $\bar q$ and $\overline{V} \subset \Omega$.
    Thus, the point $ \bar q $ must lie in the interior of $\Omega$.
    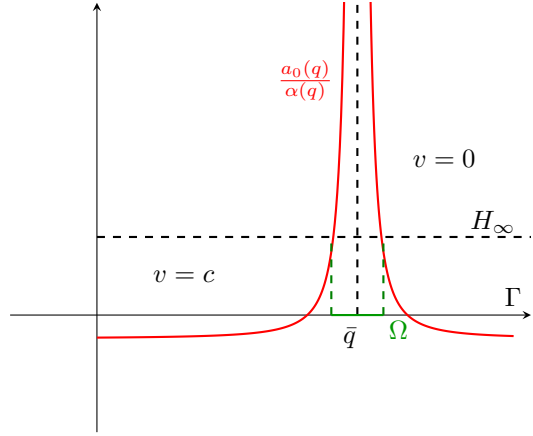
\begin{figure}[h]
        \centering
        \input{fig7.tex}
        \caption{Graphical explanation of why $\bar q \in \operatorname{int} \Omega$.}
    \end{figure}
\end{proof}

On the contrary, the second case $\Sigma = \Sigma^- = \{\bar q\}$ is more delicate.
\begin{thm}
    \label{thm:single-martinet}
    Suppose that $ \Sigma = \Sigma ^-= \{\bar q\}$, that is, there is a single Martinet point on the curve $\Gamma$ and we have $\omega_{\bar q}(X_0(\bar q)) < 0$.
    Then, letting $\kappa$ be defined as in \eqref{eq:def-k_i}, it holds the following
    \begin{equation}
        \label{eq:martinet-single-point}
        \operatorname{MC}_\varepsilon(\Gamma,T)
        \simeq
        -\frac{4}{\kappa \eps^2}
        \ln \eps,
        \quad
        \eps \to 0.
    \end{equation}
\end{thm}

\begin{rem}
    \label{rem:single-martinet}
    Notice that, in this case, the asymptotic depends just on the value $\kappa$, so it depends only on the local structure of $\Gamma$ near $\bar q$: it does not involve neither the whole curve, nor the time $T$ that we have to travel along $\Gamma$.
\end{rem}

\begin{proof}
    Let us consider the coordinates and the orthonormal frame introduced in Theorem \ref{eq:norm-coord-martinet}.
    In these coordinates, after some simplifications very similar to the one carried out in Subsection \ref{subsec:proof-of-red}, the control system \eqref{eq:cs-drift} can be expressed in cylindrical coordinates as
    \begin{equation}
        \label{eq:dot-z-martinet-cilindrical}
        \begin{cases}
            \dot r = v_1 ,               \\
            \dot \theta = \frac{v_2}{r}, \\
            \dot z = a_0(z) + \frac{r}{2}\gamma(z)v_2 .
        \end{cases}
    \end{equation}
    Here, with again a slight abuse of notation, we write $a_0(z)$ and $\gamma(z)$ in place of $a_0(0,\cdot,z)$ and $\gamma(0,\cdot,z)$.
    As we showed in the case of a step 2 distribution, these simplifications in the control system affect the asymptotic of $\MCeps$ starting from the order $O(\eps^{-1})$.
    Hence, they will not change the final result that we are going to obtain.
    We omit the details about these reductions, since it would be a very close repetition of all the arguments already exposed in Subsection \ref{subsec:proof-of-red}.

    Fix $\eps>0$.
    As in the step 2 case (see Proposition \ref{prop:bounded-controls}), we can fix a uniform bound on the norm of the controls, that is, for $c>0$ we consider the minimization problems
    \begin{multline}
        \label{eq:mc-c-def-martinet}
        \operatorname{MC}_\varepsilon(\Gamma,T,c)
        =
        \frac1\varepsilon
        \inf\bigg\{
        J(u)
        \mid
        u \in L^\infty\left( [0,T];\mathbb{R}^m \right),
        \,
        \|v\|_\infty\leq c, \, v_2 \geq 0 \text{ for a.e. } t\in[0,T],
        \,
        \\
        q_u([0,T]) \subset \operatorname{Tube}_{\varepsilon}(\Gamma),
        \,
        q_u(0)=\Gamma(0),
        \,
        q_u(T)=\Gamma(S)
        \bigg\}.
    \end{multline}
    and then take the limit as $c\to +\infty$.
    Again as it was done in the proof of the step 2 case, applying Pontryagin Maximum Principle one can prove that an optimal control is the concatenation of \textit{bang} arcs,  and for each arc we have either $v_2(t)=c$ or $v_2(t)=0$.
    Now, we restrict in a sufficiently small neighbourhood of the Martinet point $\bar q$, so that the Taylor expansion in Lemma~\ref{lem:Taylor-alpha} holds, and we want to find a lower bound of the $L^1$ norm of a bang arc with $v_2(t)=c$ in this neighbourhood of $\bar q$.
    As a consequence, from now on, we restrict to consider just controls with maximal norm steering the system from some initial point $z(t_1)$ to a final point $z(t_2)$ not depending on $\eps$.

    By \eqref{eq:dot-z-martinet-cilindrical}, we have that there exists $C,K>0$ independent of $\varepsilon$ such that
    \begin{equation}
        |\dot z(t)|
        \leq
        C + \frac{c\eps}{2} (\kappa |z(t)| + \eps K)
        ,
        \qquad \forall t\in [t_1,t_2].
    \end{equation}
    So, applying Gronwall's Lemma, we obtain
    \begin{equation}
        \label{eq:1-pto-martinet-gronwall}
        |z(t_2)|
        \leq
        \left(
        \frac{\tilde C + \frac{c \eps^2}{2}K}{\frac{c\eps}{2}\kappa}
        +
        |z(t_1)|
        \right)
        e^{\frac{c\eps}{2}\kappa (t_2-t_1)}
        =
        \left(
        \frac{2\tilde C }{c\eps \kappa}
        +
        \frac{\eps K}{\kappa}
        +
        |z(t_1)|
        \right)
        e^{\frac{\eps}{2}\kappa\|v\|_1}
        ,
    \end{equation}
    where we have used that $c(t_2-t_1)=\|v\|_1$.
    Since $c\to+\infty$ and $\eps>0$ is fixed, we can assume $c\eps^2>1$, so that $2\tilde C / c\eps \kappa \leq 2\tilde C \eps/\kappa$.
    Taking the logarithm on both side and neglecting lower order terms, we obtain
    \begin{equation}
        \|v\|_1\ge
        \frac{2}{\kappa\varepsilon}\ln \left( \frac{\kappa}{\eps \tilde K}|z(t_2)| \right)
        =-\frac{2}{\kappa \eps} \ln \eps
        +
        O\left(\frac 1\eps \right),
    \end{equation}
    which is the estimate we were looking for.

    Notice that, by this argument, we already know that the complexity of the curve containing a Martinet point is at least of order $O(-\frac{\ln \eps}{\eps^2})$, regardless of the boundary conditions and the final time $T>0$.
    Hence, fixing a neighbourhood $V_{\bar q}$ of $\bar q$ we can reduce to find an asymptotic of $\MCeps(\Gamma\cap V_{\bar q},T)$ in place of $\MCeps(\Gamma,T)$, without caring about exact boundary conditions and precise time $T$.
    This is because, from Theorem \ref{thm:main}, we know that we can go from $\Gamma(-S)$ to a point close to the Martinet point with cost $O(\frac{1}{\eps^2})$, which is asymptotically less than $O(-\frac{\ln \eps}{\eps^2})$, then cross the region near the Martinet point and, finally, we can get to our desired final point $\Gamma(S)$ again with cost $O(\frac{1}{\eps^2})$.

    The following claim completes the proof:
    \begin{claim}
        \label{claim:martinet-opt-synth}
        The asymptotical optimal synthesis for the motion planning problem restricted to a neighbourhood of the Martinet point is made by three arcs:
        \begin{enumerate}
            \item a first \emph{bang} arc, with $v_1(t)=0$ and $v_2(t)=c$, starting from a point $q_0=(\eps,\theta_0,z_0)$, with $z_0 < 0$, reaching a point $q_1=(\eps,\theta_1,z_1)$, with $z_1>0$ and $z_1\leq \tilde C\eps^3$, for some constant $\tilde C>0$.
            \item a second \emph{bang} arc with $v_1(t)=c$ and $v_2(t)=0$ on a short time interval, steering the system from $q_1=(\eps,\theta_1,z_1)$ to $q_2=(\eps,-\theta_1,z_2)$, where $z_1>z_2>0$.
            \item a third \emph{bang} arc, with $v_1(t)=0$ and $v_2(t)=c$, starting from a point $q_3=(\eps,-\theta_1,z_2)$ and reaching a point $q_4=(\eps,\theta_3,z_3)$, where $z_3$ is independent of $\eps$.
        \end{enumerate}
        Here, $c>C\varepsilon^{-4}$ for some constant $C>0$.
        The $L^1$ norm of the control of the first and the third \emph{bang} arc can be bounded above by $-\frac{2}{\kappa\eps}\ln \eps$, while the norm of the second arc is of lower order.
    \end{claim}

    \begin{figure}
        \centering
        \input{disegnoMartinet.tex}
        \caption{Graphical representation of the trajectory $z$ of point 1. in Claim~\ref{claim:martinet-opt-synth}: the blue line is the trajectory $z$, the green circle is the limit cycle of $z$, the red circle represent the distribution $\Delta$, which is transverse to the limit cycle.}
    \end{figure}
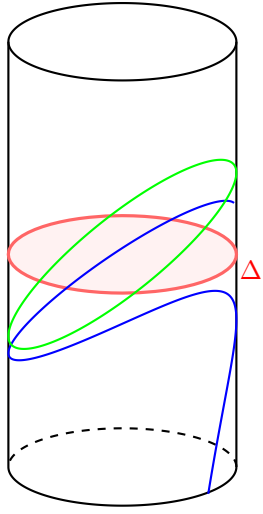

    Since we know that optimal trajectories are made by concatenations of \emph{bang} arcs and that, for $\varepsilon\ll1$, the \emph{bang} arcs with $v_2(t)=0$ do not yield trajectories crossing the Martinet surface, we have that indeed the optimal synthesis must begin, as in point 1, with a bang arc with $v_2(t)=c$.
    Let us show that the maximal $z$ coordinate obtained with such a bang control starting from $q_0$ is positive and of the order $O(\eps^3)$. Moreover, we show that this can be done with $L^1$ norm of the control bounded above by $-\frac{2}{\kappa\eps}\ln \eps$.

    Since we are localized near a Martinet point, up to modify the asymptotic of the motion complexity by a $O(\eps^{-1})$ term, we can assume that the drift is constant $a_0(z)=-a$, for some $a>0$.
    By \eqref{eq:Taylor-alpha}, we have that
    $
        \al(r,\theta,z)
        =
        -\kappa z + r \tilde \al_1 \cos(\theta+\tilde\theta)
        + o(|r|+|z|)
    $, $(r,\theta,z)\in A^{-}$, and for some constants $\tilde \al_1 \in \R, \tilde\theta\in [0,2\pi)$ and $(r,z)\to 0$.
    Notice that, under the {generic type} assumption, we have that $\tilde \al_1
        \neq0$.
    Now, we can consider controls for which $r=\eps$ and $v_2=c$ constant.
    So, the solution for $\theta$ reads
    \begin{equation}
        \theta(t) = \frac{c}{\eps}t + \theta_0.
    \end{equation}
    Moreover, the equation for $z$ for such controls is
    \begin{equation}
        \label{eq:dot-z-Martinet2}
        \dot z = -a - \frac{\eps c}{2}\kappa z + \frac{\eps^2}{2} c\tilde \al_1  \cos \big( \theta(t) + \tilde \theta ).
    \end{equation}
    This equation can be integrated explicitly:
    \begin{multline}
        z(t)
        =
        e^{-\frac{\eps \kappa c t}{2}}
        \bigg[
            \: z_0
            -
            2a \, \frac{ e^{\frac{\eps \kappa c t}{2}} - 1 }{ \eps \kappa c }
            +
            \\[5pt]
        +
        \frac{ \tilde \al_1 \eps^3 }{ 4 + \kappa^2 \eps^4 }
        \left(
        \kappa \eps^2
        \left(
            e^{\frac{\eps}{2}\kappa c t} \cos \theta(t)
            -
            \cos \theta_0
            \right)
        +
        2\big(
            e^{\frac{\eps}{2}\kappa c t} \sin \theta(t)
            -
            \sin \theta_0
            \big)
        \right)
        \bigg].
    \end{multline}
    For fixed $c$ and $t\to+\infty$, the trajectory has a limit cycle, whose equations are:
    \begin{equation}
        \label{eq:limit-cicle}
        r=\eps,
        \quad
        z
        =
        - \frac{ 2a }{ \eps \kappa c }
        +
        \frac{ \tilde \al_1 \eps^3 }{ 4 + \kappa ^2 \eps^4 }
        \left(
        \kappa \eps^2
        \cos \theta
        +
        2 \sin \theta
        \right).
    \end{equation}
    Notice that, since $\tilde\al_1\neq0$, the limit circle is transverse to the Martinet surface $\{\al=0\}$.
    The maximum value of $z$ is attained at $\sin \theta \simeq 1$ and this is positive for every $\eps>0$ if and only if
    \begin{equation}
        - \frac{ 2a }{ \eps \kappa c }
        +
        \frac{ \tilde \al_1 \eps^3 }{ 2 }
        \geq
        0.
    \end{equation}
    Hence, we must choose $c\ge C\varepsilon^{-4}$ for $C=8a/(\tilde\alpha_1\kappa)$, otherwise the corresponding trajectories cannot have $z>0$.
    So, for fixed $\eps,t_1>0$ and $c$ sufficiently large, starting from $z_0<0$, we obtain that there is a point of the resulting trajectory arbitrarily close to the point of the limit circle above the Martinet point, that is, we have $z(t_1)>0$.
    Notice that, by \eqref{eq:limit-cicle}, $z(t_1)\leq \tilde \al_1 \eps^3 /2$.

    Now, we have to estimate the $L^1$ norm of this control up to $t_1$.
    Recall that the solution to \eqref{eq:dot-z-Martinet2} with constant control is
    \begin{align}
        z(t)
         & =
        e^{\frac{-\eps \kappa c t}{2}} z_0
        -
        a \int_0^t e^{\frac{\eps \kappa c (t-s)}{2}}\, ds
        +
        \frac{\tilde \alpha_1 \eps^2}{2}
        \int_0^t e^{\frac{\eps \kappa c (t-s)}{2}} \cos (\theta(s)-\tilde \theta)\, ds
        \\
         & \ge
        e^{\frac{-\eps \kappa c t}{2}}z_0
        +
        \left( e^{\frac{-\eps \kappa c t}{2}} - 1 \right)
        \left(
        \frac{2a}{\kappa c\varepsilon} + \frac{\tilde \alpha_1 \varepsilon}{\kappa c}
        \right),
    \end{align}
    where we have used $\cos (\theta(s)-\tilde \theta)\geq -1$.
    So, choosing $z_0=-\frac{1}{2}(\frac{2a}{\kappa c \eps} + \frac{\tilde \al_1 \eps}{\kappa c})$, we obtain
    \begin{equation}
        \left(
        z(t) + \frac{2a + \tilde \al_1 \eps^2}{\kappa c \eps}
        \right)
        \frac{\kappa c \eps}{2a + \tilde \al_1 \eps^2}
        \geq
        \exp\left(-\frac{\kappa c \eps t}{2}\right),
    \end{equation}
    and taking the logarithm on both sides, for $t=t_1$ we have
    \begin{equation}
        \ln \left(
        z(t_1)
        \frac{\kappa c \eps}{2a }
        + 1
        \right)
        \geq
        -\frac{\kappa c \eps t_1}{2}.
    \end{equation}
    Neglecting lower order terms, we obtain
    \begin{equation}
        \label{eq:ineq-Martinet-costo-finale-prima-metà}
        \|v\|_1 = c t_1 \leq -\frac{2}{\kappa \eps}\ln \eps.
    \end{equation}
    Notice that this is the $L^1$ norm of the control $v$ required to reach a point $z(t)=O(\eps^3)$, $z(t)>0$.
    In particular, since this depends only on the product $ct_1$, up to further enlarge $c$ we can take $t_1$ as small as we want.
    This completes the proof of point 1 in the Claim.

    We have to prove point 2 and 3.
    By means of a constant control of the form
    $
        v_1(t)=c, \, v_2(t)=0
    $
    on a time interval of length $2\eps/c$, we can go from a point $(\eps, \theta,z_1)$ to $(\eps,-\theta,z_1-a\frac{2\eps}{c})$.
    The $L^1$ norm of this control is $2\eps$, so it is negligible in our asymptotic.
    If $c$ is big enough, then $z_1-a\frac{2\eps}{c} > 0$, hence this piece of trajectory crosses the Martinet surface.
    Thus, also point 2 follows.
    Finally, reversing the strategy applied to point 1, one can reach a point whose $z$ coordinate is positive and independent of $\varepsilon$, and the required $L^1$ of the control is the same cost as in Equation
    \eqref{eq:ineq-Martinet-costo-finale-prima-metà}.
    Thus, the total cost to cross the region near the Martinet point is the r.h.s. of \eqref{eq:ineq-Martinet-costo-finale-prima-metà} multiplied by 2, which yields the asymptotic \eqref{eq:martinet-single-point} in the Theorem.

    {The fact that the switching from point 2 to point 3 yields an asymptotically optimal synthesis follows by the same argument as the one used to justify the switch from point 1 to point 2, applied to the backward problem obtained by considering the drift $-X_0$ and the curve $\Gamma$ with reverse orientation. }
\end{proof}

\subsection{Case of many Martinet Points}
\label{subsec:many-Martinet-points}
In this Section we generalize Theorem \ref{thm:single-martinet} to the case of finitely many Martinet points, completing the proof of Theorem \ref{thm:main2}.

Let $\Gamma$ be our curve transverse to the distribution $\Delta$, $(\Delta,\Gamma)$ being a couple of generic type, and denote by $q_1\coloneqq \Gamma(s_1), \dots, q_r\coloneqq \Gamma(s_r)$ the points in $\Sigma^-$.
We have that $\al(q_i)=0$, $i=1,\dots,r$, where $\al$ was defined in \eqref{eq:alpha}.
Recall also the definition of $\kappa_1,\dots, \kappa_r$, see Equation \eqref{eq:def-k_i}.
We have to prove that
\begin{equation}
    \label{eq:equiv-mc-many-Martinet}
    \operatorname{MC}_\varepsilon(\Gamma,T)
    \simeq
    \sum_{i=1} ^r
    -\frac{4}{\kappa_i \eps^2}
    \ln \eps,
    \quad
    \eps \to 0.
\end{equation}
First, we show that the r.h.s. is a lower bound of the motion complexity and then we find a sequence of controls realizing the asymptotic.

Take any control $u$ admissible for $\MCeps(\Gamma,T)$ and denote by $q_u$
the corresponding trajectory.
Let $V_i$ be an open neighbourhood of $q_i$, $\overline{V_i}\subset \operatorname{Tube}_\eps (\Gamma)$, for every $i=1,\dots,k$. Let $I_{i}\coloneqq q_u ^{-1}(V_i)$.
Then, since the asymptotic in \eqref{eq:martinet-single-point} does not depend on the exact initial and final points of the curve that we are approximating, the curve ${q_u}_{|I_{i}}$ is an admissible competitor of the motion complexity of $\Gamma\cap V_i$. Then, by Theorem \ref{thm:single-martinet}, we must have
\begin{equation}
    \int_{I_{i}} |u(t)|dt
    \geq
    - \frac{4}{\kappa_i \eps^2} \ln \eps.
\end{equation}
On the other hand, we have that the motion complexity of $\Gamma$ away from Martinet points is of order $O(\frac{1}{\eps^2})$.
Thus, putting everything together, we obtain
\begin{equation}
    \label{eq:many-mart-points}
    \frac{1}{\eps}
    \int_0 ^T |u(t)|dt
    =
    \frac{1}{\eps}
    \sum_{i=1} ^k \int_{I_{i}} |u(t)|dt + O \left( \frac{1}{\eps^2} \right) \geq
    \sum_{i=1} ^k
    - \frac{4}{\kappa_i \eps^2} \ln \eps
    + O \left( \frac{1}{\eps^2} \right),
\end{equation}
and taking the infimum on the left hand side, we obtain the inequality we were looking for.

To conclude, it is sufficient to notice that, applying the procedure already used in the proof of Theorem \ref{thm:single-martinet}, it is possible to find a sequence of controls $(u_j)_{j\in\N}$ admissible for $\MCeps(\Gamma,T)$ such that
\begin{equation}
    \lim_{j\to +\infty}
    \int_{I_{i}} |u_j(t)|dt
    =
    - \frac{4}{\kappa_i \eps^2} \ln \eps.
\end{equation}
Thus, replacing $u$ with such $u_j$ in \eqref{eq:many-mart-points} and in taking the limit as $j\to +\infty$, we obtain the desired result.

This concludes the proof of Theorem \ref{thm:main2}. \qed

\appendix

\section{Intrinsic definition of $\kappa$}

Let $\omega$ be a form associated with $(\Delta,\Gamma)$. This form defines a vector field $W_\Gamma$ on $\Gamma$ such that $W_\Gamma(q)\in T_q\Gamma$ and $\omega(W_\Gamma)=1$.  Pick any point $q_i\in \Sigma$, then the quantities $\kappa_i$ are defined in \cite{c1} as
\begin{equation}
    \kappa_i = \left|\varphi'(0)\right|,
    \qquad\text{where}\qquad
    \varphi(t):=\frac{d}{dt}\bigg|_{t=0} \omega\left([X_1,X_2](e^{tW_\Gamma}(q_i))\right).
\end{equation}
Let us show that these coincide with the expression in \eqref{eq:def-k_i}.
\begin{lem}
    \label{lemma:ciao}
    Consider any extension of $W_\Gamma$ in a neighbourhood of $\Gamma$ such that $\omega(W_\Gamma)\equiv 1$.
    Then, for any $q_i\in \Sigma$ we have
    \begin{equation}
        \label{eq:ciao}
        \varphi'(0)=
        \omega([W_\Gamma,[X_1,X_2]])(q_i) - d\omega(W_\Gamma,[X_1,X_2])(q_i).
    \end{equation}
    In particular, this value is independent on the choice of the extension.
\end{lem}

\begin{proof}
    By Cartan formula, we have
    \begin{align}
        \label{eq:bah}
        \frac{d}{dt}\bigg|_{t=0} \omega\left([X_1,X_2](e^{tW_\Gamma}(q_i))\right)
        =
        \left.
        d\left( i_{[X_1,X_2]} \omega \right) (W_\Gamma)
        \right|_{q_i}
        =
        \\
        =
        \left(
        \mathcal L_{[X_1,X_2]}\omega(W_\Gamma) - i_{[X_1,X_2]}d\omega(W_\Gamma)
        \right)|_{q_i}
    \end{align}
    By definition of Lie derivative, we get
    \begin{equation}
        \begin{split}
            \mathcal L_{[X_1,X_2]}\omega(W_\Gamma)
             & = [X_1,X_2](\omega(W_\Gamma)) - \omega([[X_1,X_2],W_\Gamma]) \\
             & = \omega([W_\Gamma,[X_1,X_2]]).
        \end{split}
    \end{equation}
    Here, we used the fact that $\omega(W_\Gamma)\equiv 1$ and thus $[X_1,X_2](\omega(W_\Gamma))=0$. Finally, \eqref{eq:ciao} follows from the fact that $i_{[X_1,X_2]}d\omega(W_\Gamma) =-d\omega(W_\Gamma,[X_1,X_2])$.

    The fact that the resulting value is independent on the choice of the extension is a consequence of the definition of $\kappa_i$. Let us anyway check it directly by considering a different extension $\tilde W$. Since $\omega(\tilde W_\Gamma) = \omega(W_\Gamma)$ we have that $\tilde W_\Gamma = W_\Gamma + Y$ for some $Y\in \ker \omega=\Delta$. Then, the statement follows by linearity of the objects under consideration and the fact that $\omega([Y,[X_1,X_2]])=d\omega(Y,[X_1,X_2])$ since both $Y$ and $[X_1,X_2]$ belongs to $\Delta$ at $q$.
\end{proof}

\begin{rem}
    \label{rem:def-kappa-intrinsic}
    If instead of $\omega$ we consider $\tilde \omega = \varphi \omega$, for any $\varphi > 0$, we have
    \begin{align}
        \tilde \kappa (q_i)
         & =
        \left. \tilde W_\Gamma(\varphi) \right|_{q_i} \omega\left([X_1,X_2]\right)
        +
        \varphi (q_i) \frac{d}{dt}\bigg|_{t=0+} \omega\left([X_1,X_2](e^{t \tilde W_\Gamma}(q_i))\right)
        \\[5pt]
         & =
        \Big(
        \varphi  \tilde W_\Gamma \Big( \omega([X_1,X_2]) \Big)
        \Big)\Big|_{q_i}
        =
        \lim_{t\to0+}
        \Big(
        W_\Gamma \big( \omega([X_1,X_2]) \big)
        \Big)\Big|_{q_i}
        =
        \kappa(q_i),
    \end{align}
    where we have used $\omega\left([X_1,X_2]\right)=0$ at $q$ (i.e., $q$ is a Martinet point) and the identity $\tilde W = \frac{1}{\varphi}W$.
    So, $\kappa$ does not depend on the choice of $\omega$.
\end{rem}

\bibliographystyle{alpha}
\bibliography{biblio}

\end{document}

%% file: fig2.tex
\begin{tikzpicture}
    \begin{axis}[
        grid=none,
        ticks=none,
        ymin=-1,ymax=8,
        xmin=-2,xmax=8.5,
        axis lines = middle,
        xlabel=$t$,ylabel=$z$,
        label style = {at={(ticklabel cs:1.1)}}
        ]
        \draw[color=black, thick]  (0,0) -- (1,2) ;
        \draw[color=black, thick, name path=line 1]  (1,2) -- (3,3) ;
        \draw[color=black, thick]  (3,3) -- (5,5) node[right, yshift=0.1cm]{$z_v(T)$};
        \draw[color=black, dashed]  (1,0) node[below, xshift=0.1cm]{\small $t_*$} -- (1,2) ;
        \filldraw [black] (5,0) circle (1pt) node[below]{$T$} ;
        \filldraw [black] (0,5) circle (1pt) node[left]{$z_f$} ;
        \draw[color=black, thick, dashed]  (0,5) -- (5,5) ;
        \draw[color=black, thick, dashed]  (5,0) -- (5,6.25) node[right]{\small $z_0(T)$} ;
        \draw[color=black, thick, dashed]  (0,6.25) -- (5,6.25) ;
        \draw[color=black,domain=1:5,smooth, thick, dotted, name path=line 2] plot ( \x, {0.15*(\x-1)^2+2} )  node[right]{\small $z_{\tilde v}(T)$};
        \draw [red,thick] plot [smooth, tension=1] coordinates { (0.2,0.4) (3,2.2) (5,5)};
        \draw [green,dashed] plot [smooth, tension=1] coordinates { (0.5,1) (3,2.7) (5,5.75)};
        \filldraw [green] (5,5.75) circle (1.5pt) node[right]{\small $z_0(T-t;z_v(t))$};
        \draw[color=green, dashed]  (0.5,0) node[below]{\small $t$} -- (0.5,1) ;
        \draw[color=black, dashed]  (0.2,0) node[below]{\small $t_1$}-- (0.2,0.4) ;
        \draw[green, dashed] (0,1) node[left]{\small $z_v(t)$} -- (0.5,1);
        \filldraw [green] (0.5,1) circle (1.5pt);
        \filldraw [red] (0.2,0.4) circle (1.5pt);
        \filldraw [red] (5,5) circle (1.5pt);
    \end{axis}
\end{tikzpicture}

%% file: fig3.tex
\begin{tikzpicture}
    \begin{axis}[
        grid=none,
        ticks=none,
        ymin=-1,ymax=8,
        xmin=-1,xmax=7,
        axis lines = middle,
        xlabel=$t$,ylabel=$z$,
        label style = {at={(ticklabel cs:1.1)}}
        ]
        \draw[color=black, thick]  (0,0) -- (1,2) ;
        \draw[color=black, thick, name path=line 1]  (1,2) -- (3,3) ;
        \draw[color=black, thick]  (3,3) -- (5,5) node[right]{$z_v(T)$} ;
        \draw[color=red, thick, dashed]  (1,2) -- (1,0) node [below]{$t_*$};
        \filldraw [black] (5,0) circle (1pt) node[below]{$T$} ;
        \filldraw [black] (0,5) circle (1pt) node[left]{$z_f$} ;
        \draw[color=black, thick, dashed]  (0,5) -- (5,5) ;
        \draw[color=black, thick, dashed]  (5,0) -- (5,7) ;
        \draw[color=black,domain=0:5,smooth, thick, dashed] plot ( \x, {0.25*\x^2} ) node[right]{$z_0(T)$};
        \draw[color=black, thick, dashed]  (0,6.25) -- (5,6.25) ;
        \draw[color=black,domain=1:5,smooth, thick, dotted, name path=line 2] plot ( \x, {0.3*(\x-1)^2+2} )  node[above]{$z_{\tilde v}(T)$};
        \fill[red,name intersections={of=line 1 and line 2,by={A,E}}]
            (intersection-2) circle (2pt) ;
        \draw[color=red, thick, dashed]  (E) -- (E|-0,0) node [below]{$t_2$};
    \end{axis}
\end{tikzpicture}

%% file: fig4.tex
\begin{tikzpicture}
    \begin{axis}[
        grid=none,
        ticks=none,
        ymin=-1,ymax=8,
        xmin=-1,xmax=6.5,
        axis lines = middle,
        xlabel=$t$,ylabel=$z$,
        label style = {at={(ticklabel cs:1.1)}}
        ]
        \draw[color=black, thick]  (0,0) -- (1,2) ;
        \draw[color=black, thick]  (1,2) -- (3,3) ;
        \draw[color=black, thick, name path=segm 1]  (3,3) -- (5,6.25) node[right, yshift=2mm]{$z_v(T)$};
        \draw[color=red, thick, dashed]  (1,2) -- (1,0) node [below]{$t_*$};
        \filldraw [black] (5,0) circle (1pt) node[below]{$T$} ;
        \draw[color=black, thick, dashed]  (0,5) -- (5,5) ;
        \draw[color=black, thick, dashed]  (5,0) -- (5,7) ;
        \draw[color=black,domain=0.001:5,smooth, thick, dashed] plot ( \x, {2.236*sqrt{\x}} ) node[right,yshift=2mm]{$z_0(T)$};
        \draw[color=black, thick, dashed]  (0,6.25) -- (5,6.25) ;
        \filldraw [black] (0,6.25) circle (1pt) node[left]{$z_f$};    
        \draw[color=black,domain=1.001:5,smooth, thick, dotted, name path=segm 2] plot ( \x, {2*sqrt{\x-1}+1} )  node[right]{$z_{\tilde v}(T)$};
        \fill[red,name intersections={of=segm 1 and segm 2,by={F}}]
             (intersection-1) circle (2pt) ;
        \draw[color=red, thick, dashed]  (F) -- (F|-0,0) node [below]{$t_2$};
    \end{axis}
\end{tikzpicture}

%% file: fig5.tex
\begin{tikzpicture}
    \begin{axis}[
        grid=none,
        ticks=none,
        ymin=-1,ymax=8,
        xmin=-2,xmax=8.5,
        axis lines = middle,
        xlabel=$t$,ylabel=$z$,
        label style = {at={(ticklabel cs:1.1)}}
        ]
        \draw[color=black, thick]  (0,0) -- (1,2) ;
        \draw[color=black, thick, name path=line 1]  (1,2) -- (3,3) ;
        \draw[color=black, thick]  (3,3) -- (5,6.25) node[right, yshift=0.1cm]{$z_v(T)$};
        \draw[color=black, dashed]  (1,0) node[below, xshift=0.1cm]{\small $t_*$} -- (1,2) ;
        \filldraw [black] (5,0) circle (1pt) node[below]{$T$} ;
        \draw[color=black, thick, dashed]  (0,5) -- (5,5) node[right]{\small $z_0(T)$};
        \draw[color=black, thick, dashed]  (5,0) -- (5,7.5) ;
        \draw[color=black, thick, dashed]  (0,6.25) -- (5,6.25) ;
        \filldraw [black] (0,6.25) circle (1pt) node[left]{$z_f$};    
        \draw[color=black,domain=1:5,smooth, thick, dotted, name path=line 2] plot ( \x, {4.5*(sqrt{\x-1})+2} )  node[right]{\small $z_{\tilde v}(T)$};
        \draw [red,thick] plot [smooth, tension=1] coordinates { (0.2,0.4) (2,4.5) (5,6.25)};
        \draw [green,dashed] plot [smooth, tension=1] coordinates { (0.5,1) (3,4) (5,5.75)};
        \filldraw [green] (5,5.75) circle (1.5pt) node[right]{\small $z_0(T-t;z_v(t))$};
        \draw[color=green, dashed]  (0.5,0) node[below]{\small $t$} -- (0.5,1) ;
        \draw[color=black, dashed]  (0.2,0) node[below]{\small $t_1$}-- (0.2,0.4) ;
        \draw[green, dashed] (0,1) node[left]{\small $z_v(t)$} -- (0.5,1);
        \filldraw [red] (0.2,0.4) circle (1.5pt);
        \filldraw [red] (5,6.25) circle (1.5pt);
        \filldraw [green] (0.5,1) circle (1.5pt);
    \end{axis}
\end{tikzpicture}

%% file: fig6.tex
\begin{tikzpicture}
    \begin{axis}[
        axis lines = middle,
        xlabel = $z$,
        ylabel = {$\frac{a_0(z)}{\al(z)}$},
        samples = 200,
        xmin = -1,
        xmax = 3,
        ymin = -1.5,
        ymax = 4,
        grid=none,
        xtick=\empty, 
        ytick=\empty        
    ]
        \addplot[red, thick, domain=0:2.85] {(x-1)^3 - 3*(x-1)+1 };
        \draw[dashed,thick] (0,1) -- (3,1) node[yshift=2mm,xshift=-10mm]{$H$};
        \node at (1,2) {$v=0$};
        \node at (1.5,0.5) {$v=c$};
    \end{axis}
\end{tikzpicture}

%% file: fig7.tex
\begin{tikzpicture}
    \begin{axis}[
        axis lines = middle,
        xlabel = $\Gamma$,
        samples = 200,
        xmin = -1,
        xmax = 5,
        ymin = -1.5,
        ymax = 4,
        grid=none,
        xtick=\empty, 
        ytick=\empty        
    ]
        \addplot[red, thick, domain=0:2.9] {1/(10*(x-3)^2) - 0.3 };
        \addplot[red, thick, domain=3.1:4.8] {1/(10*(x-3)^2) - 0.3 };
        \draw[dashed,thick] (3,0) node[yshift=-3mm, xshift=-1mm]{$\bar q$} -- (3,4) ;
        \draw[green!50!black,dashed,thick] (2.7,0) -- (2.7,1) ;
        \draw[green!50!black,dashed,thick] (3.3,0) -- (3.3,1) ;
        \draw[dashed,thick] (0,1) -- (5,1) node[yshift=2mm,xshift=-5mm]{$H_\infty$};
        \draw[green!60!black,thick] (2.7,0) -- (3.3,0) node[yshift=-2mm,xshift=2mm]{$\Omega$};
        \node at (4,2) {$v=0$};
        \node at (1,0.5) {$v=c$};
        \node at (2.4,3) { \color{red} $\frac{a_0(q)}{\al(q)}$} ;
    \end{axis}
\end{tikzpicture}

%% file: disegnoMartinet.tex
\tdplotsetmaincoords{70}{110}

\begin{tikzpicture}[tdplot_main_coords, scale=1.5]

\def\R{1}          
\def\Zmax{2}    
\def\H{4}       


\foreach \theta in {110,290}{
    \draw[thick]
        ({\R*cos(\theta)}, {\R*sin(\theta)}, -\Zmax)
        -- ({\R*cos(\theta)}, {\R*sin(\theta)}, \Zmax);
}

\draw[thick,dashed]
    plot[domain=110:290, samples=200]
    ({\R*cos(\x)},{\R*sin(\x)},-\Zmax);
\draw[thick]
    plot[domain=-70:110, samples=200]
    ({\R*cos(\x)},{\R*sin(\x)},-\Zmax);

\draw[thick]
    plot[domain=0:360, samples=200]
    ({\R*cos(\x)},{\R*sin(\x)},\Zmax);

\filldraw[color=red!60, fill=red!5, very thick](0,0,0) circle (\R) ;
\node at (0,1.2,0) {\color{red} $\Delta$};

\draw[blue, thick, samples=300, domain=8:1.2, variable=\x]
    plot
    ({cos(\x r)}, {sin(\x r)}, {-(1.7/\x)^2+(1-1/\x)*(0.5*cos(\x r)+sin(\x r))-0.2});

\draw[thick,green]
    plot[domain=0:360, samples=200]
    ({\R*cos(\x)},{\R*sin(\x)},{0.5*cos(\x)+sin(\x)});

\end{tikzpicture}

%% file: main-hal-v2.bbl
\begin{thebibliography}{RMGMP04}

\bibitem[ABB19]{ABB}
Andrei Agrachev, Davide Barilari, and Ugo Boscain.
\newblock {\em A Comprehensive Introduction to Sub-Riemannian Geometry}.
\newblock Cambridge Studies in Advanced Mathematics. Cambridge University Press, 2019.

\bibitem[ABS08]{Agrachev2008}
A.~Agrachev, U.~Boscain, and M.~Sigalotti.
\newblock A {G}auss-{B}onnet-like formula on two-dimensional almost-{R}iemannian manifolds.
\newblock {\em Discrete Contin. Dyn. Syst.}, 20(4):801--822, 2008.

\bibitem[Bel96]{Bellaiche1996}
A.~Bella{\"i}che.
\newblock The tangent space in sub-{R}iemannian geometry.
\newblock In {\em Sub-{R}iemannian geometry}, volume 144 of {\em Progr. Math.}, pages 1--78. Birkh\"auser, Basel, 1996.

\bibitem[BG13]{Boizot}
Nicolas Boizot and Jean-Paul Gauthier.
\newblock On the motion planning of the ball with a trailer.
\newblock {\em Mathematical Control and Related Fields}, 3(3):269--286, 2013.

\bibitem[Gau12]{gauthierbonnard}
J.-P. Gauthier.
\newblock Motion planning for kinematic systems, 2012.
\newblock Preprint, in honor of Bernard Bonnard.

\bibitem[GJZ09]{14}
J.-P. Gauthier, B.~Jakubczyk, and V.~Zakalyukin.
\newblock Motion planning and fastly oscillating controls.
\newblock {\em SIAM J. Control Optim.}, 48(5):3433--3448, 2009.

\bibitem[GK14]{gauthierKawskiz}
Jean-Paul Gauthier and Matthias Kawskiz.
\newblock Minimal complexity sinusoidal controls for path planning.
\newblock In {\em 53rd IEEE Conference on Decision and Control}, pages 3731--3736, 2014.

\bibitem[GZ05a]{9}
J.-P. Gauthier and V.~Zakalyukin.
\newblock On the codimension one motion planning problem.
\newblock {\em J. Dyn. Control Syst.}, 11(1):73--89, 2005.

\bibitem[GZ05b]{10}
J.-P. Gauthier and V.~Zakalyukin.
\newblock On the one-step-bracket-generating motion planning problem.
\newblock {\em J. Dyn. Control Syst.}, 11(2):215--235, 2005.

\bibitem[GZ05c]{11}
J.-P. Gauthier and V.~Zakalyukin.
\newblock Robot motion planning, a wild case.
\newblock In {\em In Proceedings of the 2004 Suszdal conference on dynamical systems, volume 250 of Proceedings of the Steklov Mathematical Institute}, 2005.

\bibitem[GZ06]{12}
J.-P. Gauthier and V.~Zakalyukin.
\newblock On the motion planning problem, complexity, entropy, and nonholonomic interpolation.
\newblock {\em J. Dyn. Control Syst.}, 12(3):371--404, 2006.

\bibitem[GZ07]{13}
J.-P. Gauthier and V.~Zakalyukin.
\newblock Entropy estimations for motion planning problems in robotics.
\newblock In {\em Volume In honor of Dmitry Victorovich Anosov, Proceedings of the Steklov Mathematical Institute}, volume 256, pages 62--79, 2007.

\bibitem[Jea01]{Jean2001}
F.~Jean.
\newblock Complexity of nonholonomic motion planning.
\newblock {\em Internat. J. Control}, 74(8):776--782, 2001.

\bibitem[Jea03]{Jean2003a}
F.~Jean.
\newblock Entropy and complexity of a path in sub-{R}iemannian geometry.
\newblock {\em ESAIM Control Optim. Calc. Var.}, 9:485--508 (electronic), 2003.

\bibitem[Jea14]{jeanControl2014}
Fr{\'e}d{\'e}ric Jean.
\newblock {\em Control of {{Nonholonomic Systems}}: From {{Sub-Riemannian Geometry}} to {{Motion Planning}}}.
\newblock {{SpringerBriefs}} in {{Mathematics}}. Springer International Publishing, Cham, 2014.

\bibitem[JP15]{c2}
F.~Jean and D.~Prandi.
\newblock Complexity of control-affine motion planning.
\newblock {\em SIAM Journal on Control and Optimization}, 53(2):816--844, 2015.

\bibitem[Pra13]{prandi2013}
D.~Prandi.
\newblock H\"older continuity of the value function for control affine system.
\newblock Accepted by ESAIM, Control Optim. Calc. Var., 2013.

\bibitem[RMGMP04]{c1}
C.~Romero-Mel{\'e}ndez, J.~P. Gauthier, and F.~Monroy-P{\'e}rez.
\newblock On complexity and motion planning for co-rank one sub-{R}iemannian metrics.
\newblock {\em ESAIM Control Optim. Calc. Var.}, 10(4):634--655, 2004.

\end{thebibliography}
